\documentclass[12pt]{amsart}
\usepackage{latexsym,fancyhdr,amssymb,color,amsmath,amsthm,graphicx,listings,comment}
\usepackage[section]{placeins}  \newtheorem{thm}{Theorem} \newtheorem{lemma}{Lemma}
\newtheorem{coro}{Corollary}  \let\paragraph\subsection
\def\Binomial#1#2{{#1\choose #2}}

\title{Manifolds from Partitions}
\author{Oliver Knill}
\date{1/24, 2024, updated 2/1/2024}
\address{Department of Mathematics \\ Harvard University \\ Cambridge, MA, 02138 }
\begin{document}
\maketitle

\begin{abstract}
If $f$ maps a discrete $d$-manifold $G$ onto 
a $(k+1)$-partite complex $P$ then $H(G,f,P)$, the set of simplices $x$ 
in $G$ such that $f(x)$ contains at least one facet in $P$ 
is either a $(d-k)$-manifold or empty.
\end{abstract} 

\section{The theorem}

\paragraph{}
A {\bf $d$-manifold} is a finite abstract simplicial complex $G$ for which all unit
spheres $S(x)$ are $(d-1)$-spheres. The {\bf join} $A \oplus B$ of simplicial complexes is
the disjoint union $A \cup B \cup \{ x \cup y, x \in A, y \in B \}$. Spheres form a sub-monoid
of all complexes as $S_{A \oplus B}=S_A(x) \oplus B$ or $A \oplus S_B(x)$. 
An {\b f integer partition} $p=(n_0,\dots, n_k)$ of $n=n_0+\cdots n_k$ defines 
the {\bf partition complex} $P=K_{\{(n_0,\dots,n_k)\}}$. It is the join of finitely many 
$0$-dimensional complexes $K_{(n_j)}=\overline{K_{n_j}}$ and a {\bf $(k+1)$-partite graph} 
of maximal dimension $k$. Partition complexes form a submonoid of all complexes. 
A map $f$ from $V(G)$ to $V(P)$ is {\bf P-onto} if its image is $k$ dimensional.
Given a $d$-manifold $G$ and a $P$-onto map $f:G \to P$, define
$H(G,f,P)$ as the pull back of $P$ under $f$: the set $U$ of all simplices $y \in G$ 
such that $f(y)$ contains one of the facets in $P$
is an open set \cite{Alexandroff1937,FiniteTopology} and defines the
vertex set of a graph, where two $a,b \in U$ are connected if one is contained in the other. 
$H(G,f,P)$ is the simplicial complex defined by the complete sub-graphs.
We always consider maps $f:G \to P$ defined  by point maps $f: V(G) \to V(P)$. They
are continuous with respect to the Alexandrov topologies: the inverse of a sub-simplicial complex
in $P$ is a sub-simplicial complex in $G$. A $(k+1)$ partition $p$ on $[1,n]=\{1,,\dots,n\}$ 
defines so a $k$-dimensional topology on $[1,n]$ and so an irreducible representation 
of the symmetry group of $[1,n]$. 

\begin{thm}
\label{1}
If $f$ maps a $d$-manifold $G$ P-onto a $k$-complex $P \in \mathcal{P}$,
then $H(G,f,P)$ is a $(d-k)$-manifold or empty.
\end{thm} 
\begin{proof}
Write $H=H(G,f,P)$ and use induction with respect to $d-k$. If $d-k=0$, then 
by assumption, we get a finite set of $d$-simplices in $G$ which each are mapped 
into a maximal simplex of $P$. Since all these elements of $G$ are maximal, 
they are disconnected and form a $0$-dimensional complex, a 0-manifold. 
Let $x \in G$ be such that $f(x)$ contains a maximal simplex in $P$. 
We have to show that $S_H(x)$ is a sphere. Use the
hyperbolic structure $S_G(x) =S_G^+(x) \oplus S_G^-(x)$, where $S_G^+(x)$ 
is the set of simplices strictly containing $x$ and $S_G^-$ the set of 
simplices strictly contained in $x$. In the case of a
manifold $G$, both $S_G^+(x)$ and $S_G^-(x)$ are spheres. 
Since every $y \in S^+(x)$ is in $S^+_H(x)$, also $S_K^+(x)$
is a sphere. To see that $S^-_H(x)$ is a sphere
of co-dimension $k$ in the sphere $S^-_G(x)$ use induction and that
${\rm dim}(S^-_G(x))-k$ is smaller than $d-k$. 
The induction starts with $d-k=0$. We
additionally need to show that if $G$ is the boundary sphere of a
$(k+1)$-simplex, then $\{ f=P\}$ is $2$ point graph.
The complex $S^-(x)$ has $k+2$ vertices and $P$ is a $k+1$-partite complete graph with
$k+1$ slots. A map $\{1, \dots, k+2\}$ to $\{0,1,2, \dots, k\}$
must have exactly two points $a,b \in x$ which are mapped
into the same point. This checks that $\{f=P\}$ is a 0-sphere. 
\end{proof} 

\paragraph{}
Theorem~(\ref{1}) can be seen as a {\bf discrete inverse function theorem}. 
It is the discrete analog of the fact that if $f:M \to N$ is a smooth map from a d-manifold
$M$ to a $k$-manifold $N$ and $p$ is a regular value in $N$, then $f^{-1}(p)$ is a $d-k$-manifold.
The discrete Sard result \cite{DiscreteAlgebraicSets}
is a special case, where we used a particular partition complex (topology) on the finite 
target set $\{ -1,1\}^k$ of ${\rm sign}(f-c)$ given by a map 
$f: G \to \mathbb{R}^k$. Theorem~(\ref{1}) implies in particular that if
$f:G \to y=\{0,\dots, k\}$ is surjective, then the set of simplices $x$ with $f(x)=y$ is a 
$(d-k)$-manifold or empty. In \cite{KnillSard}, we looked maps $f:G \to \mathbb{R}$ which 
leads for $c$ not in $f(G)$ to $x \to {\rm sign}(x) \in \{-1,1\}$.
If $f:G \to \mathbb{R}$ has no root and takes both positive and negative
values, then $\{ f=0 \}$ is always a hyper-surface. This was already magic 
in the very special case $P=K_2$. 

\begin{figure}[!htpb]
\scalebox{0.02}{\includegraphics{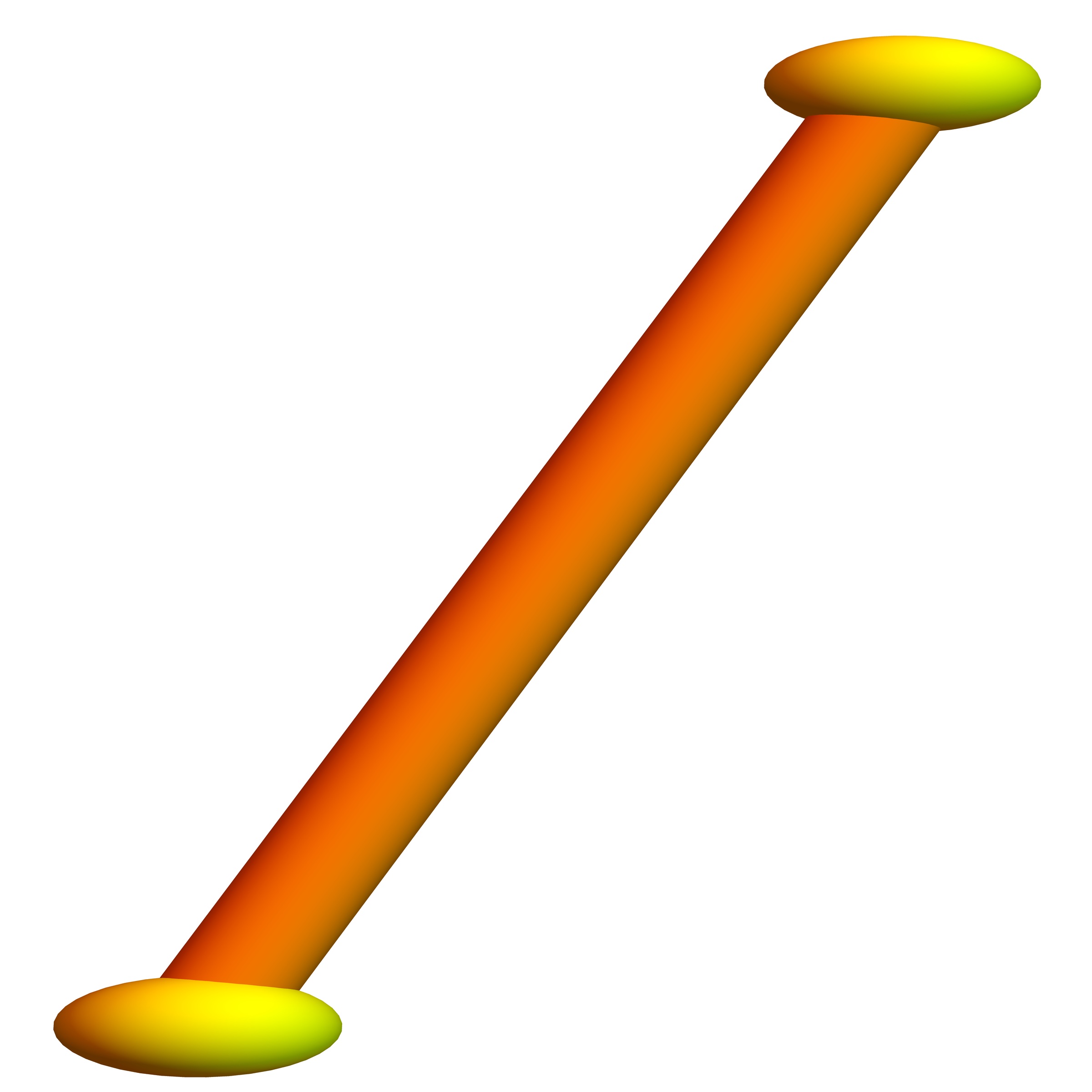}}
\scalebox{0.02}{\includegraphics{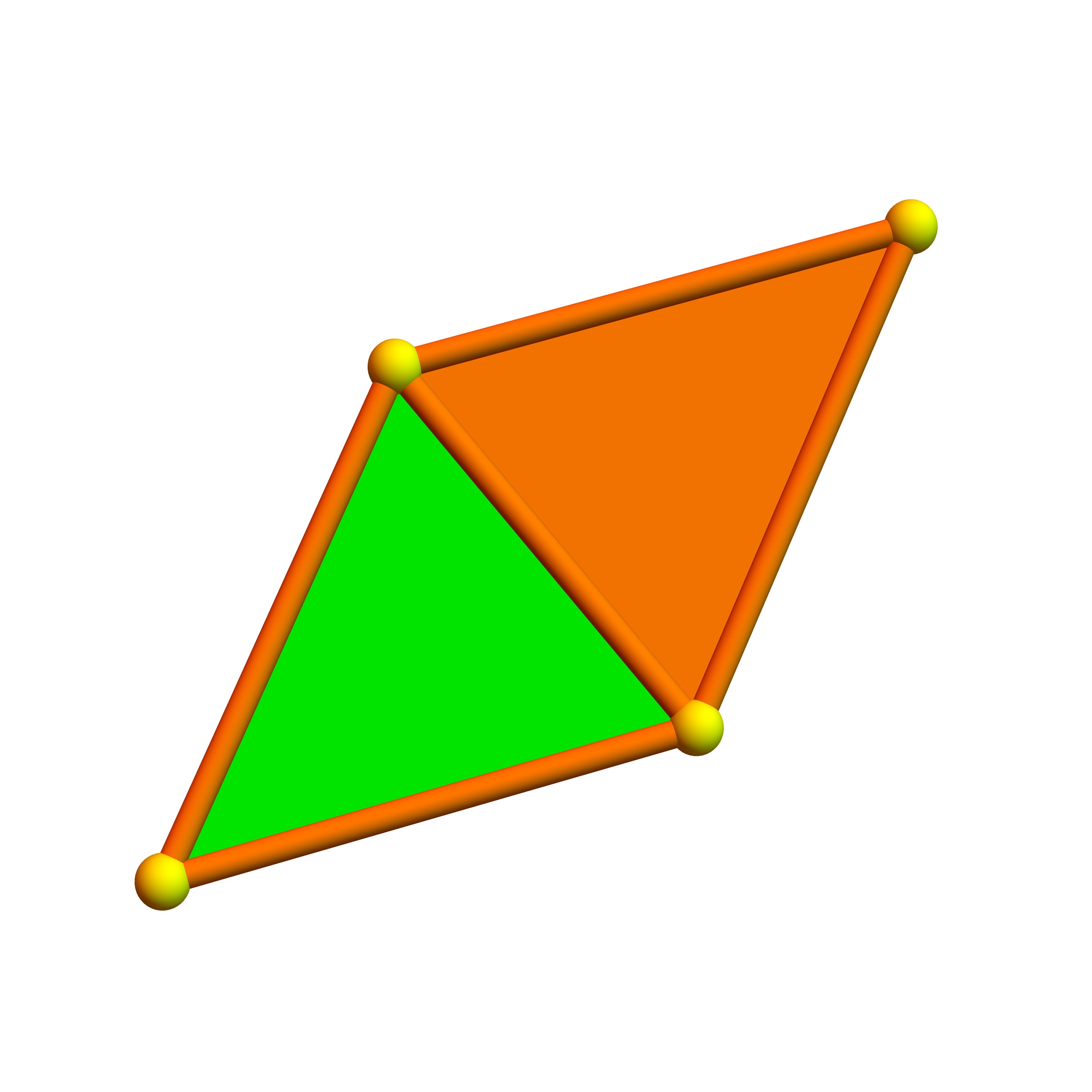}}
\scalebox{0.02}{\includegraphics{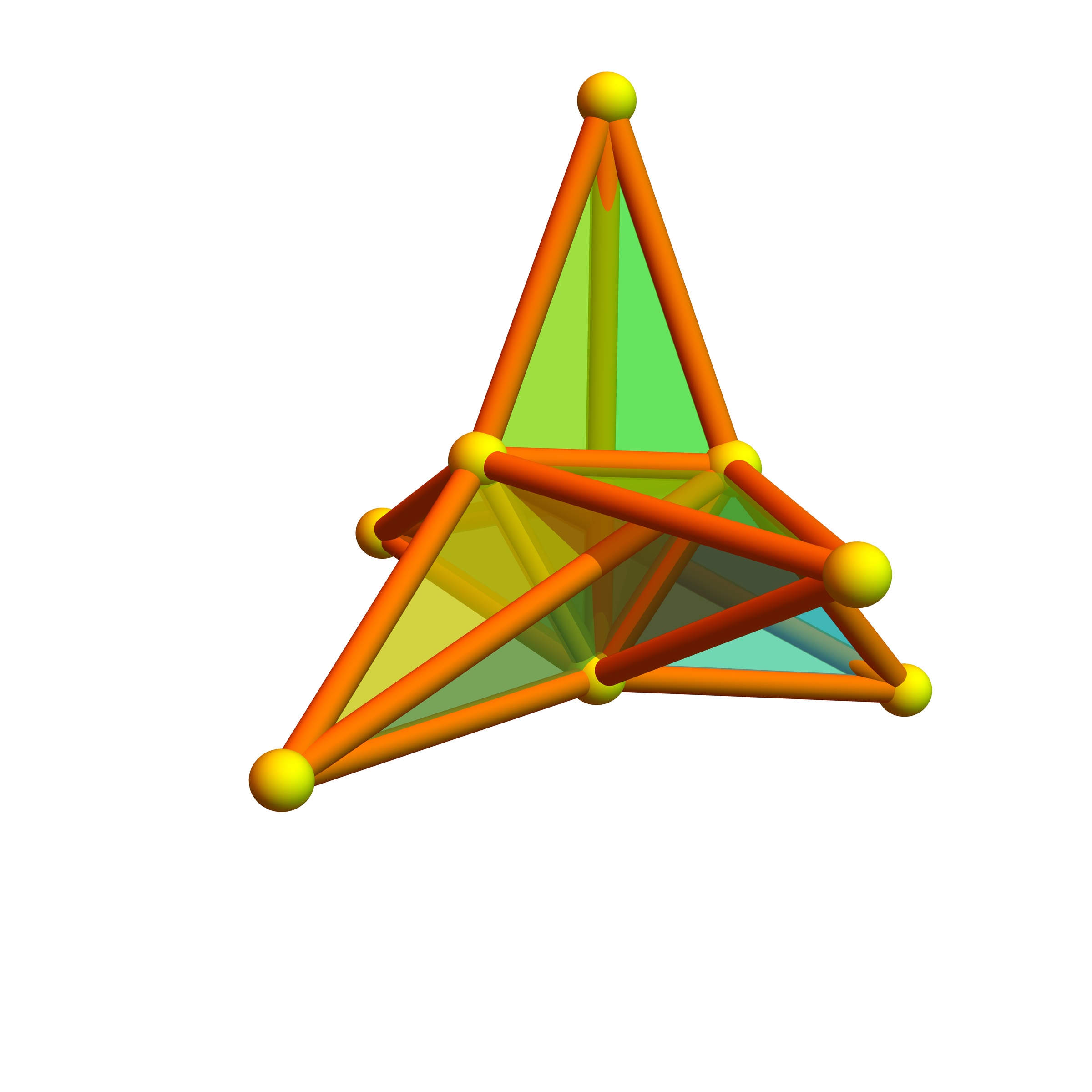}}
\caption{
For Sard $G \to \mathbb{R}$, where studying $\{ f =c \}$ is determined
by the map ${\rm sign}(f-c)$ onto $\{(-1,1)\}$, the complex
$K_2=K_{(1,1)}$ is the only relevant choice \cite{KnillSard}.  
For Sard $G \to \mathbb{R}^2$, where ${\rm sign}(f-c)$  map into 
$\{(-1,1),(1,1),(-1,-1),(-1,-1)\}$, we needed to pick the unique 2-dim 
partition complex $P=K_{(1,1,2)}$ that exists for $n=4$
\cite{DiscreteAlgebraicSets}. For $G \to \mathbb{R}^3$, the map
${\rm sign}(f-c)$ targets $\{-1,1\}^3$ which has 8 points. 
The 3-dimensional partitions of $n=8$ are 
$(5,1,1,1), (4,2,1,1),(3,3,1,1),(3,2,2,1), (2,2,2,2)$. 
In \cite{DiscreteAlgebraicSets}, we picked $p=(1,1,1,5)$ by forced the target signs
to be $(-1,1,1),(1,-1,1),(1,1,-1)$ and requiring  to reach a 4th point. 
This complex $P$ is displayed to the right. For maps into $\mathbb{R}^m$
we chose $p=(1,1,\dots, 1,2^m-m)$ in \cite{DiscreteAlgebraicSets}.
}
\end{figure}

\paragraph{}
For the classical Sard theorem, see \cite{Morse1939,Sard1942}. The discrete case
$f: G \to \mathbb{R}$ can be understood by looking at maps from $G$ to $\{-1,1\}$. This was
understood in 2015 \cite{KnillSard}. 
The tri-partite case came from the situation $(f,g):G \to \mathbb{R}^2$
where four different sign cases $\{ (-1,1),(1,1)$, $(-1,-1),(1,-1) \}$ are relevant.
We noticed in the fall of 2023 that the kite complex $P=K_{(1,1,2)}$ works 
\cite{DiscreteAlgebraicSets}. It was natural because there is exactly one partition 
of $n=4$ with dimension 2: $p=(1,1,2)$. Having looked at the $6$ different sub kite graphs 
of $K_4$ we started to wonder which maps $f$ from $G$ to any simplicial complex $P$ produces 
sub-manifolds: when does $f:G \to P$ always give sub-manifolds $H(G,f,P)$?
After experimenting with hundreds of target complexes $P$, it became clear around new 
years eve 2023/2024, that {\bf partition complexes} $\mathcal{P}$ do the job.  

\paragraph{}
Having a surjective map from a $d$-manifold to $K_2$ is equivalent to 
look for $\{ f=0\}$ for $G \to \mathbb{R}$ as we only need
to distinguish positive and negative values. This produces
a hyper-surface. For a surjective map to $P=K_3=K_{(1,1,1)}$, we get a 
co-dimension-$2$-manifold. The case $P=K_{(1,1,2)}$, the kite graph,
appears if $f$ takes values in $\mathbb{R}^2$, which means that there are
4 different sign values. The condition imposed in that paper paraphrases that
we impose the $K_{1,1,2}$ structure. We can however also look at maps
from a 4-manifold to $\{1,2,3,4,5\}$ and impose the $K_{(1,1,3)}$ structure on 
this set. $P=K_{(1,1,3)}$ is the windmill graph. An other possibility with 
a target set $\{1,2,3,4,5\}$ is the wheel graph $W_4=K_{(2,2,1)}$. In the 
case $n=6$, we also can get $K_{(2,2,2)}$, the {\bf Octahedron complex}. 

\paragraph{}
An interesting case is when the target complex
is $P=K_{(2,2,2,2)}$ which is a {\bf 3-sphere}, the 3-dimensional {\bf cross polytop},
one of the regular polytops in ambient dimension $4$. We can implement the 
{\bf quaternion group} $Q$ on $P$, as we can identify the $8$ points in $P$ with 
$1,-1,i,-i,j,-j,k,-k$. In other words, a map $f$ from $G$ to the quaternion group 
$Q$ which not have an Abelian image, produces a co-dimension $3$ manifold $H(G,f,P)$, 
where $P$ is the 3-sphere simplicial complex structure on $Q$ coming from the partition 
$p=(2,2,2,2)$. In particular, any map from a 4-manifold to $Q$ with this 
property produces a co-dimension $3$ manifold
$H(G,f,P)$ which must be a finite collection of 1-dimensional cycle graphs. 

\begin{coro}
If $G$ is a $d$-manifold and $f: G \to \{0,1,\dots,k\}$ is surjective
then $H$ is a $(d-k)$-manifold or empty.
\end{coro}
\begin{proof}
This is the special case $P=K_{(1,1,\dots, 1)} = K_{k+1}$
which belongs to the decomposition $k+1=1+1+\cdots+1$. In 
representation theory, this partition of $k+1$ is related to the 
sign representation.
\end{proof} 

\paragraph{}
The empty case is rare. For the smallest 3-sphere $G=K_{(2,2,2,2)}$ 
and the triangle $P=K_{(1,1,1)}$, there are
$5796$ surjective maps. There are only 24 cases, where $H$ is empty. They
are the maps to $\{1,2,3\}$ where only one number appears more than once.

\paragraph{} 
If $P$ is a $k$-dimensional partition complex and $y$ is a facet in $P$
which is in the image of $f$, then $f^{-1}(y)$ is a $(d-k)$-{\bf manifold
with boundary} or empty. This allows us to construct without much effort manifolds with 
boundary or pictures of {\bf cobordism examples}:
if the boundary of the manifold is split into two sets $A,B$, then $A,B$ are
by definition cobordant. In illustrations done before, like for the 
expository article \cite{knillcalculus} in 2012, we needed to 
construct by hand the "hose manifold". We later would use triangulations
which numerical methods would provide when parametrizing surfaces $\vec{r}(u,v)$
where $(u,v) \subset \mathbb{R}^2$ is a closed region in the parameter plane. 
The code provided in the code section is shorter than what one needs to get
in order to rely on numerical schemes used by computer algebra
systems for drawing surfaces. The code works for manifolds of arbitrary dimension. 

\begin{coro}
Every $(d-k)$-manifold $H(G,f,P)$ is the union of patches $f^{-1}(y)$, each
of which is a $(d-k)$-manifold with boundary. 
\end{coro}
\begin{proof}
We can follow the same argument as in the theorem. 
If $x$ is in $f^{-1}(y)$, we again look at $S_G(x)=S_G^-(x) + S_G^+(x)$.
Again $S_H^+(x) = S_G^+(x)$, but now it is possible that $S_H^-(x)$ is a 
either a hypersphere in the sphere $S^-(x)$ or then co-dimension $1$ manifold
with boundary which means it is a ball. Induction brings us down to the 
case where $S_H^-(x)$ is either a 2 point complex or a 1 point complex. 
In the former case, we have an interior point, in the later a
boundary point. 
\end{proof} 

\begin{figure}[!htpb]
\scalebox{0.08}{\includegraphics{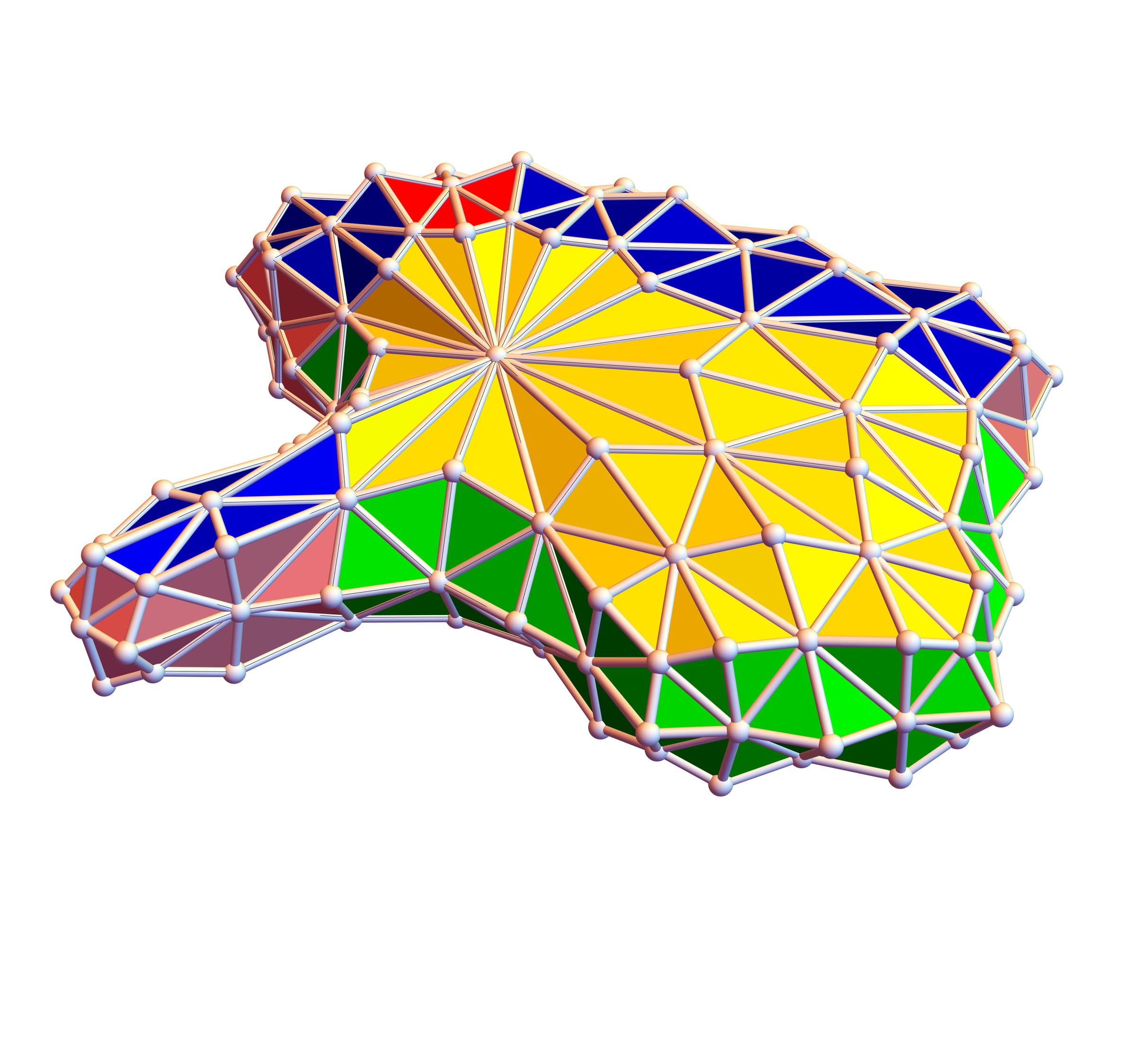}}
\caption{
As a 4-manifold $G=K_{(2,2,2)} \oplus C_{10}$ we took the join 
of the Octahedron complex with a circular complex. The
2-dimensional partition complex $P$ is by $p=(1,2,3)$ and a random function
from $G$ to $\{1,\dots,6\}$.  Now we can look 
at the $6$ facets $y$ of $P$ and color each of these patches
$f^{-1}(y)$ differently. The 6 patches cover the manifold. The 
number of patches for $n=(n_0, \dots, n_k)$ is $\prod_{j=0}^k n_j$. 
}
\end{figure}

\section{Partitions}

\paragraph{}
The semi-ring of {\bf partition complexes} $\mathcal{P}$ contains the semi-ring of integers 
and is contained in the Sabidussi ring $\mathcal{G}$ of all graphs, 
where the addition is the Zykov join \cite{Zykov} and the multiplication is the
large multiplication \cite{Sabidussi}. See \cite{HammackImrichKlavzar} for graph 
multiplications. The semi-ring is dual to the Shannon semi-ring, where the addition is 
the disjoint union
and the multiplication is the strong product \cite{Shannon1956}. Each of the rings
can be augmented to rings by group completing the additive monoid structure.
In the Zykov-Sabidussi picture, the ``integers" $\mathbb{Z}$ are given by complete 
graphs $K_n$. Because the join $K_n \oplus K_m = K_{n+m}$ and the large 
multiplication $K_n \otimes K_m = K_{n m}$ the map $n \to K_n$ is an 
isomorphism of $\mathbb{Z}$ to $\mathcal{Z}$. 
When written as partition, $K_n = K_{(1,1,\dots, 1)}$. The complex $K_{(n)}$ is
classically denoted with $\overline{K_n}$. 

\paragraph{}
For partitions, $n=(n_0, \dots, n_k)$, $m=(m_0, \dots, m_l)$, 
the addition is the concatenation $n+m = (n_0, \dots, n_k,m_0, \dots, m_l)$. 
Multiplication is product 
$n*m = (n_0 m_0, n_1, m_0, \dots, n_k m_0, n_1 m_0$, $\dots$, $n_1 m_l, \dots, n_k m_0 \dots, n_k m_l)$.
It obviously enhances the standard multiplication in the case $0$-dimensional case $k=0,l=0$. 
The order in a partition does not matter. The partition $(2,3)$ is the same than the partition $(3,2)$. 
For example $(3,4) + (2,1) = (3,4,2,1)$ and $(3,4)*(2,1,1) = (3,3,4,4,6,8)$ or 
$(1,1,1)*(2,3,4) = (2,2,2,3,3,3,4,4,4)$.
Also partitions form so a semi-ring $(P,+,*,0,1)$, and can be enhanced to be a ring. 
The dimension of a partition $n$ is the number of elements minus $1$.
Giving a partition on $[1,n]=\{1,\dots, n\}$ into $k+1$ sets defines a $k$-dimensional 
simplicial complex and so a k-dimensional Alexandrov topology on $[1,n]$. The maximal 
possible dimension $n-1$ is achieved, if we take the partition $(1,1,\dots, 1)$ of $n$.
In the zero dimensional case $n=(n)$, we have the discrete topology.  

\paragraph{}
Every partition $p:$  $n=(n_0,\dots,n_k)$ defines a {\bf $(k+1)$-partite graph} 
$K_p = K_{(n_0, \dots, n_k)}$ defined as the join 
$\oplus_{j=0}^k \overline{K_{n_j}}$ of the $0$-dimensional graphs
$K_{(n_j)} = \overline{K_{n_j}}$. 
A partition of dimension $k$ defines a $k$-dimensional graph in the sense that the
maximal clique has has size $k+1$. The
multiplication of partitions corresponds to the {\bf large multiplication} of graphs
first considered by Sabidussi. Any partition in which one of the elements is prime therefore is a 
multiplicative prime in this ring. The additive primes in the partition ring
are the 0-dimensional partitions $(n)$. They can not be written as a sum of smaller partitions. 

\paragraph{}
The complete $(k+1)$-partite graphs are defined by one of the $p(n,k+1)$ partitions
$n_0,\dots,n_k$ of $n$. For $n=7$ for example, 
there are $p(7)=15$ partitions over all. For $n=4$, there is only one 2-dimensional complex,
the case when $p=(1,1,2)$. For $k>0$, the graphs are (k+1)=partite graphs $K_{(n_0,\dots, n_k)}$.
For $k=0$, we have $K_{(n)} = \overline{K_n}$ which is the zero-dimensional 
graph with n-vertices.  \footnote{Mathematica identifies $CompleteGraph[\{5\}]$ with 
$CompleteGraph[\{1,1,1,1,1\}]$ which is inconsistent with the picture that
it should be the join of $k$ graphs $K_1$. Of course, one should keep the notation
$CompleteGraph[5]$ as $K_5$ as it is entrenched. We use 
$GraphComplement[CompleteGraph[{5}]]$ to get the $0$-dimensional $K_{(5)}$. }

\paragraph{}
The theory of partitions is rich. Euler already noticed that
taking products of geometric series produces the 
generating function for partitions
$\sum_n p(n) x^n = \prod_{k} (1-x^k)^{-1}$. 
The fact that partitions naturally define a ring, extending the 
ring of integers, appears to be less known. To see a partition 
as a {\bf choice of topology} on $[1,n]$ even less. 
We actually consider a partition $p$ as a partition complex $P$ 
on $[1,n]$ and so as a {\bf  topological structure}
of dimension $k$ on the ``point" $[1,n]$. A map $f:G \to P$ 
is {\bf regular}, if the image is $k$ dimensional. This notion replaces
the maximal rank condition of the Jacobian matrix $df$ of 
a smooth map $f$. An other picture is to see $P$ as a ``particle" as
Wigner proposed in 1939 already to see irreducible representations of 
groups as particles. 

\begin{figure}[!htpb]
\scalebox{0.12}{\includegraphics{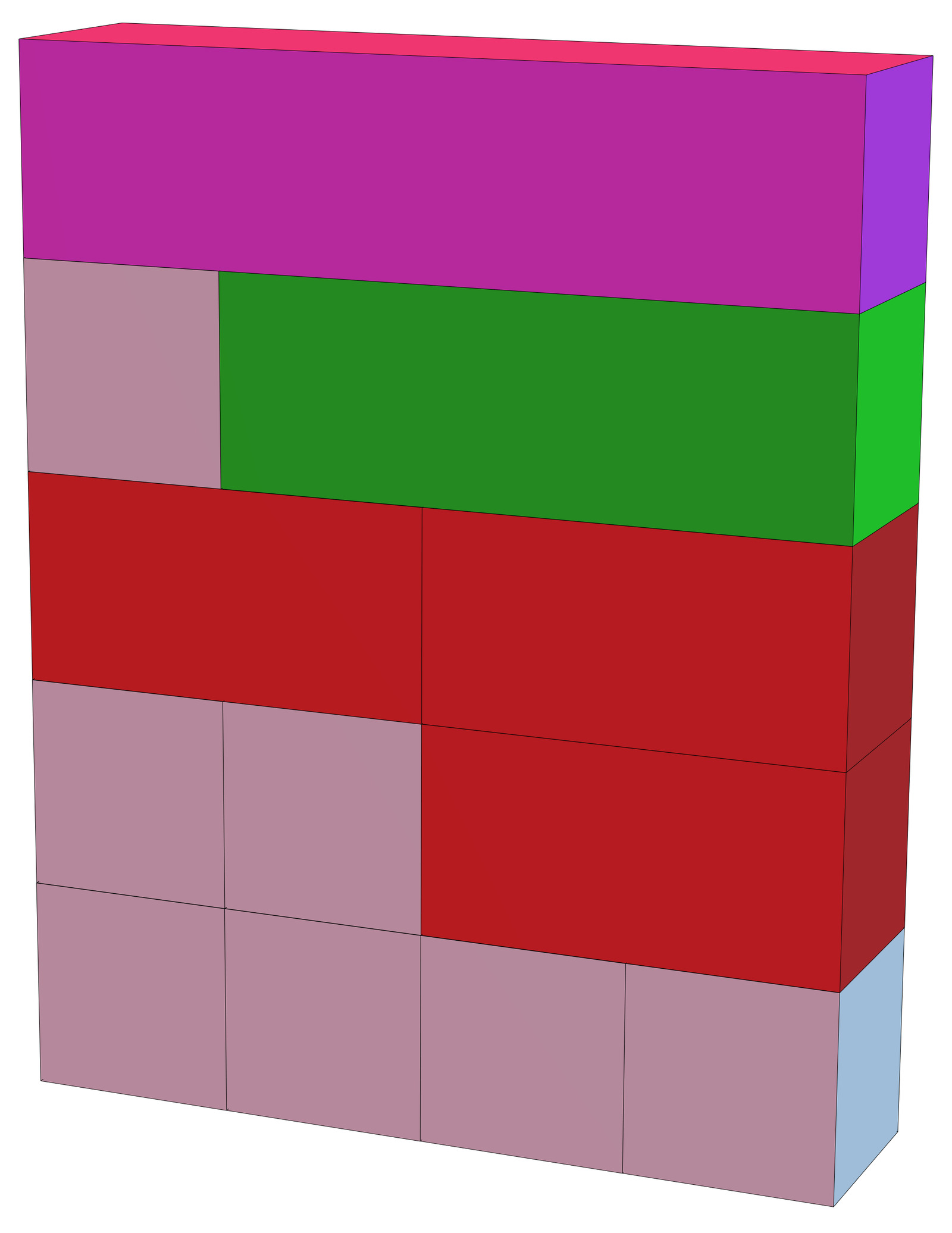}}
\scalebox{0.12}{\includegraphics{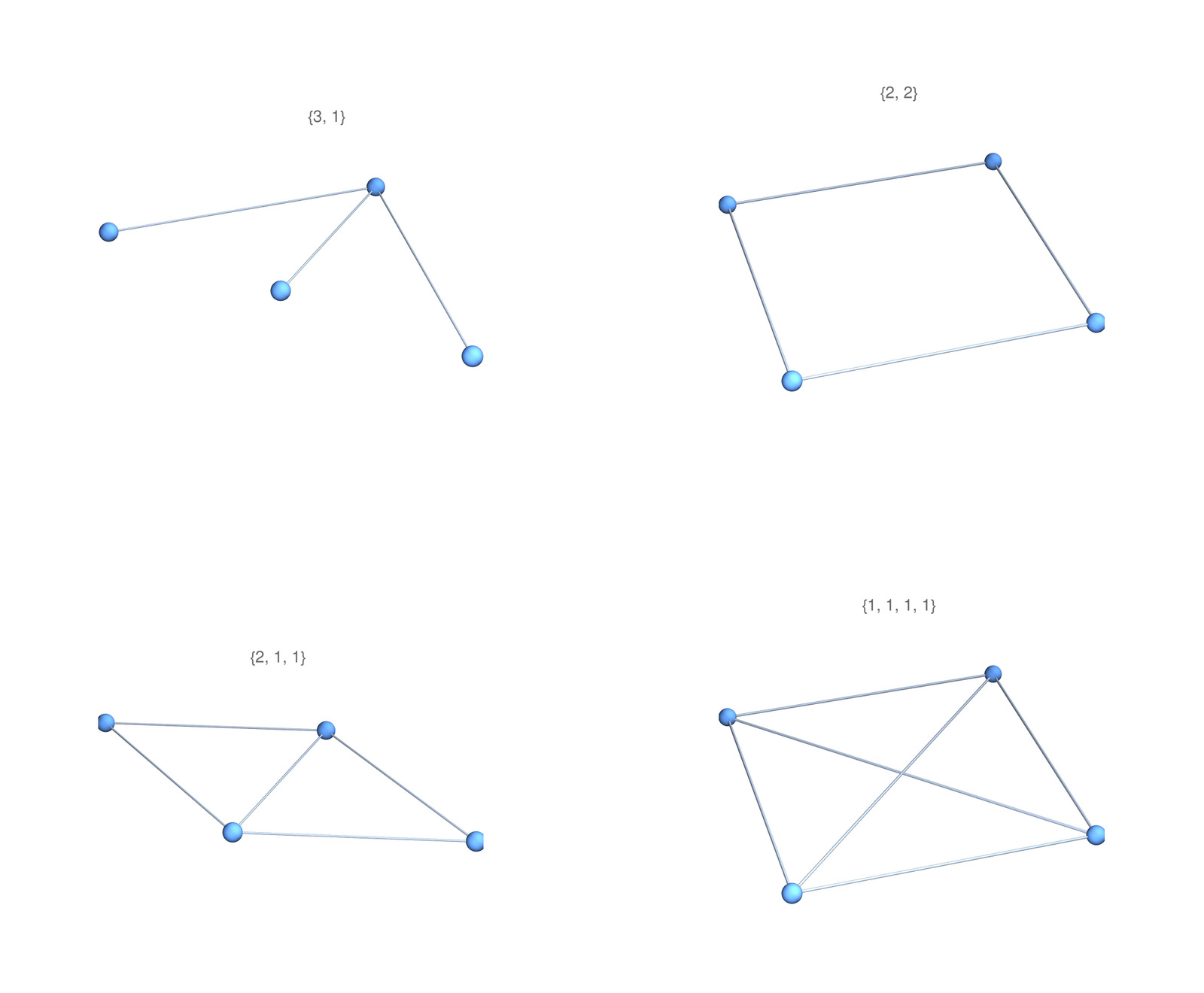}}
\caption{
We see the 5 integer partitions for $n=4$ and the corresponding
4 positive dimensional complexes. We left out 
$K_{\{5\}} = \overline{K_5}$. There is only one 2-dimensional
complex $P=K_{(1,1,2)}$, the kite complex. This complex was involved when
looking at maps $G \to \mathbb{R}^2$, where $\{ f = 0, g=0\}$ 
involved the 4 cases $\{ 1,1\},\{1,-1\},\{-1,1\},\{-1,-1\} \}$
and where we had imposed that 2 points $(1,-1),(-1,1)$ and
at least one third point are in the image $f(G)$. We rephrase this now
as that the image has to contain at least one maximal simplex in $P$. }
\end{figure}

\begin{figure}[!htpb]
\scalebox{0.12}{\includegraphics{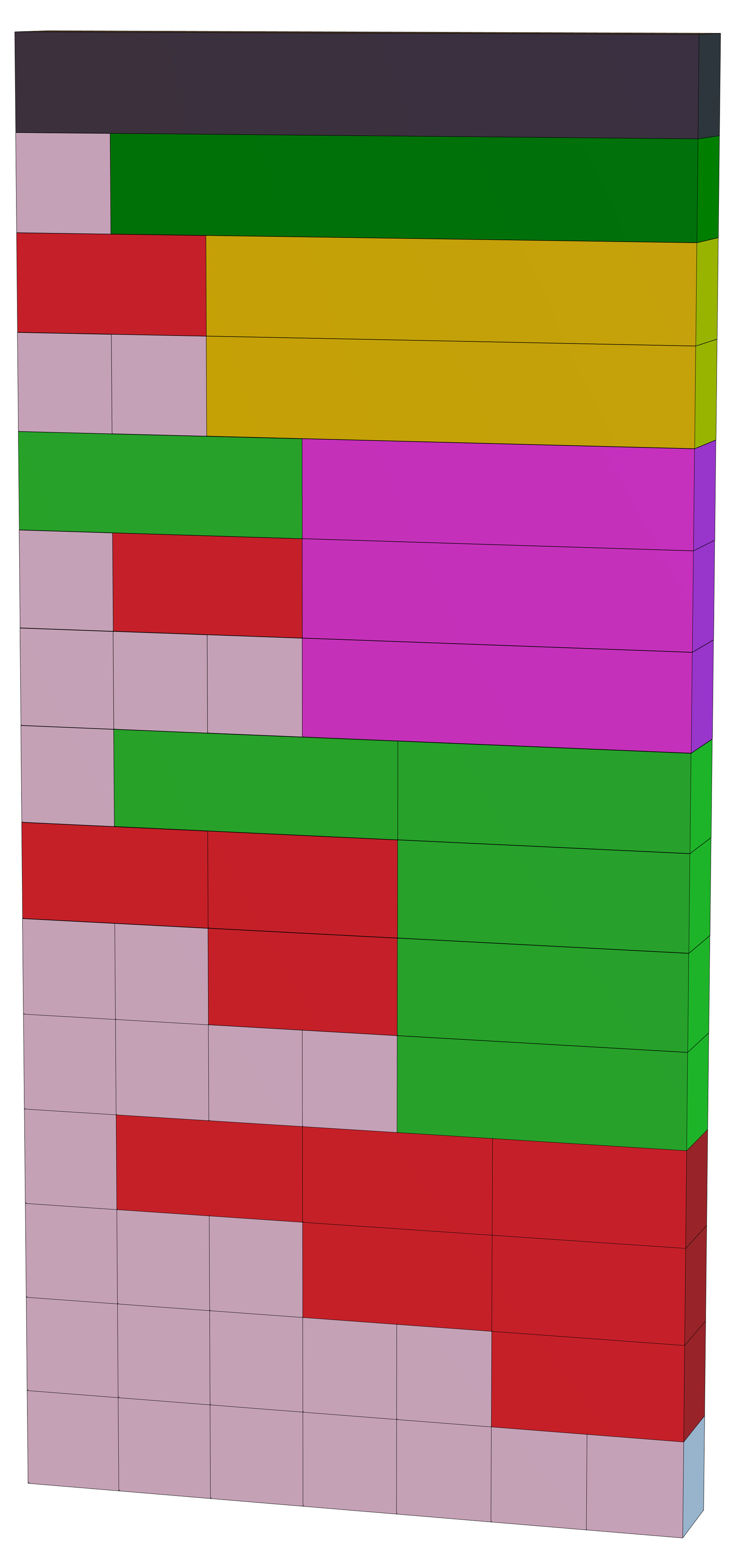}}
\scalebox{0.12}{\includegraphics{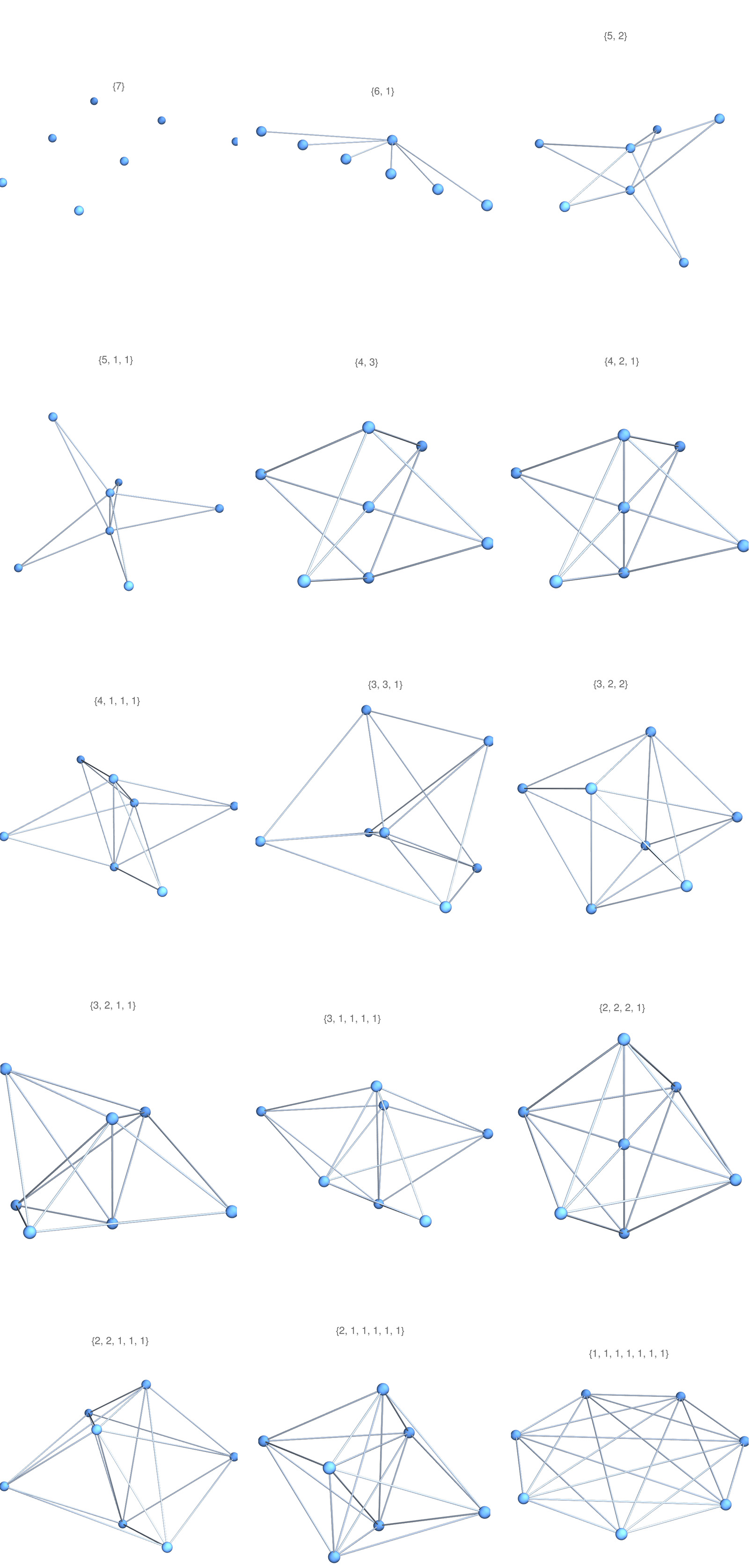}}
\caption{
All 15 integer partitions for $n=7$ and the corresponding
15 complexes. We used Cuisenaire tools in first grade 
\cite{CuisenaireGattegno} were partitions were studied for
the competition \cite{KnillSchweizerJugendForscht}. 
As far as I know, this was the first use of such a 
pedagogical tool for research rather than for teaching. 
Beside known results like the Euler pentagonal theorem, 
unusual results appeared like
like $p(n)=\sum_{k=1}^n \sigma(k) p(n-k)/n$, where
$\sigma(k)$ is the total number of prime factors of $k$.
This can be seen by looking at $n p(n)$ as the area of the
Cuisenaire rectangle listing all the partitions. 
}
\end{figure}

\paragraph{}
Lets look at some figures. The first figures
shows the case of a $5$ manifold $G=I \oplus I$,
with icosahedron 2-sphere $I$ to $P=K_{(1,1,1,1)}$. 
We took then a random map from the
vertex set $\{1,\dots,24\}$ to $\{1,2,3,4\}$ like 
$(3, 4, 1, 2, 2, 3, 4, 2, 2, 2, 2, 1, 4, 4, 
4, 2, 1, 4, 3, 1, 3, 4, 4, 4)$ for the picture. 
Each of the 274701298344704 surjective maps 
from $\{1,\dots,24\}$ to $\{1,2,3,4\}$ produces a 
2-manifold. There is no exception. 

\begin{figure}[!htpb]
\scalebox{0.1}{\includegraphics{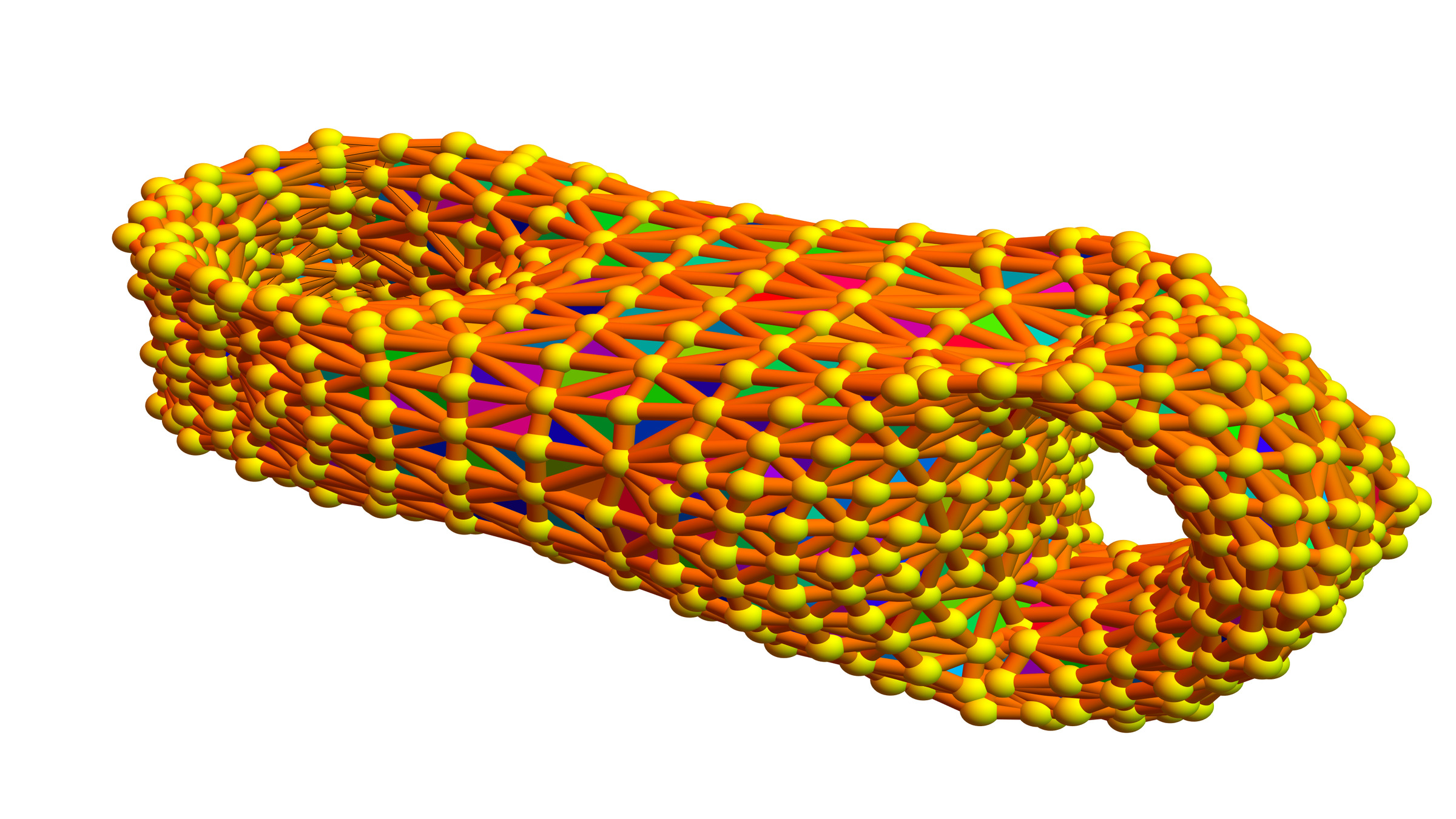}}
\caption{
A 2-manifold obtained as a co-dimension-3 manifold in the 
5-manifold $S^2 \oplus S^2$, where $S$ is an icosahedron complex.
The target space was the 3-dimensional $K_4$. In this case, we got a
genus $2$ surface $H(G,f,P)$. 
}
\end{figure}

\begin{figure}[!htpb]
\scalebox{0.2}{\includegraphics{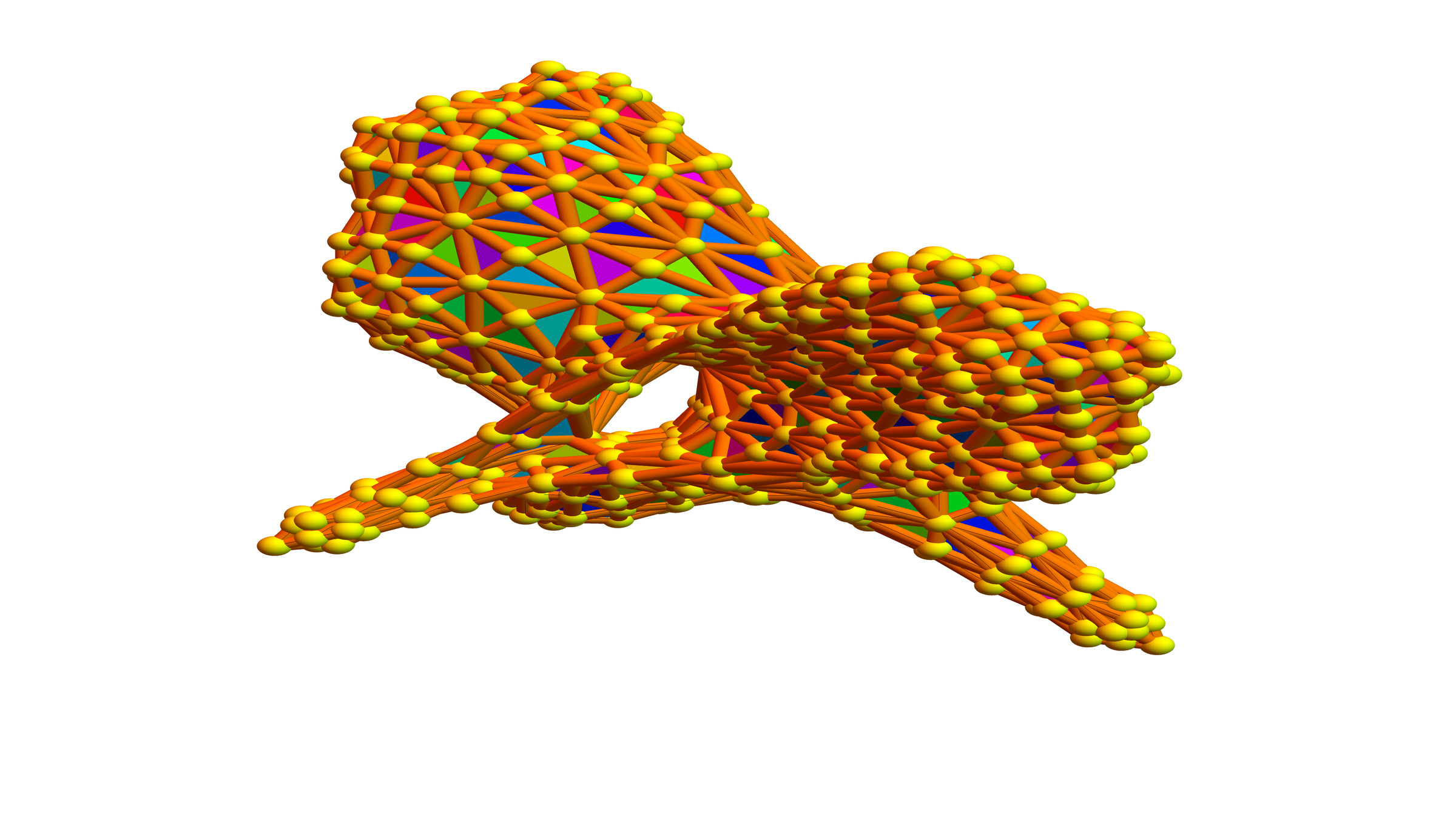}}
\caption{
A 2-manifold obtained by taking the target complex $P=K_{(2,2,3)}$ and 
$G=I \oplus C_{13}$, where $I$ is the icosahedron complex.
While the graph $G$ has only $12+13=25$ vertices, the graph $H(G,f,P)$
is much larger than $G$ as it is a sub-graph of the Barycentric refinement of $G$. 
}
\end{figure}

\paragraph{}
The additive Zykov monoid $(\mathcal{G},\oplus)$ of graphs contains various interesting 
sub-monoids. One of them is the monoid $(\mathcal{S},\oplus)$ of spheres. 
If $A$ is a $k$-sphere and $B$ is a $l$-sphere, then $A+B$ is a $k+l+1$ sphere as indicated
already in the introduction. Manifolds are finite simple graphs for which all unit spheres 
$S(x)$ are in $\mathcal{S}$. Manifolds however are not a sub-monoid and spheres do not form 
a subring. The Dehn-Sommerville monoid $\mathcal{D}$ is interesting as it generalizes spheres.
See \cite{dehnsommervillegaussbonnet}. 
The class $\mathcal{D}$ is inductively defined as a graph $G$ for which all unit spheres are 
in $\mathcal{D}$ of dimension $1$ less and such that $\chi(G) = 1+(-1)^d$, where $d$ is 
the dimension starting with the assumption that the empty graph $0$ in $\mathcal{D}$ of dimension $-1$.
Also the Dehn-Sommerville class is not invariant under multiplication.
Interesting about Dehn-Sommerville is that unlike spheres, that the definition does not invoke 
homotopy notions. 

\paragraph{}
We now look at a geometric characterization of the partition monoid $(\mathcal{P},\oplus)$. 
Let $\mathcal{Q}$ be the set of graphs which have the property that it contains $0$ and
such that it is closed under taking unit spheres: all $S(x)$ are in $\mathcal{Q}$ of
one dimension less. The Partition monoid actually agrees with this elegant definition: 

\begin{lemma}
$\mathcal{P}=\mathcal{Q}$. 
\end{lemma}
\begin{proof}
(i) Given $G \in \mathcal{P}$. Then, $G=K_{(n_0,\dots, n_k)}$. Take a vertex $x$ in $G$. It must
be located in one $K{(n_j)}$ of the zero-dimensional additive factors $G=K_{(n_0)}, \dots, G=K_{(n_k)}$. 
But then $S(x) = K_{(n_0,\dots,n_{j-1},n_{j+1},\dots, n_k)}$ which is in $\mathcal{P}$ and 
has dimension $1$ less. We have verified that $G \in \mathcal{Q}$. \\
(ii) Before showing the reverse, note that if $x,y$ are not connected in $G \in \mathcal{P}$
then $S(x) = S(y)$, implying that the diameter of $G$ is $\leq 2$ and that $G = S(x) \oplus K_{(n)}$ 
for some $n$ for every $x$. \\
(iii) Now assume $G \in \mathcal{Q}$. We show by induction with respect to dimension $d$ that it 
must be in $\mathcal{P}$. If $d=0$, then $G$ must have $(-1)$-dimensional unit spheres $S(x)$ 
meaning that every point in $G$ is isolated. Therefore $G=K_{(n_0)} = \overline{K_{n_0}}$. 
Now, assuming we have shown it to dimension $d$ and let $G$ be in $\mathcal{Q}$ of dimension
$d+1$. Then by definition of $\mathcal{Q}$, every unit sphere in $G$ has dimension $d$ and 
is in $\mathcal{Q}$. By induction assumption, $S(x)$ must be of the form $K_{(n_0,\dots, n_d)}$. 
The unit ball $B(x)$ is obtained by attaching all vertices $y \in S(x)$ to $x$ which means
$B(x) = S(x) \oplus 1 = S(x) \oplus K_{(1)}$. Therefore $B(x) =K_{(n_0,\dots, n_d,1)}$.
Take $z \in G \setminus B(x)$. It can not be connected to $x$ (as it would be in $S(x)$). 
But since $S(x)=S(y)$ shown in (ii), we have $B(y)=B(x)$.
So, $B(y) \cup B(x) = K_{(n_0,\dots, n_d,1)}$. Continue like this until no element has is any more 
available. Then $G = K_{(n_0,\dots, n_d,n)}$ for some $n$, showing that $G \in \mathcal{P}$. 
\end{proof} 

\paragraph{}
The class $\mathcal{P}$ of partition graphs has a nice recursive definition.
The graphs $\mathcal{P}$ are also characterized number theoretically in that all 
{\bf additive primes} in the monoid are zero-dimensional. In the dual Shannon picture, $\hat{\mathcal{P}}$ 
consists of disjoint union of complete graphs. If one of the additive factors is $1$, it is {\bf contractible},
meaning that there exists a vertex (actually one can take any vertex), such that $G \setminus \{x\}$
is contractible. The cohomology in general however is not trivial. For the utility graph $K_{(3,3)}$ for 
example, the Betti vector is $(1,4)$. In general, the Betti vector of $G \in \mathcal{P}$ 
is $(1,0, \dots, \prod_{j=0}^{k} (n_j-1) )$ for $k>0$ and $n_0$ for $k=0$. In particular, for
cross polytopes $p=(2,2,\dots, 2)$, it is $(1,0,\dots, 1)$. 
The {\bf inductive dimension} of $G \in \mathcal{P}$ is defined as the average of the dimensions 
of the unit spheres is equal the maximal dimension. 
They are small graphs of diameter $1$ or $2$ and have the property that given 
any finite set of points, then the intersection of all these unit spheres is again 
in the class. We should think of elements in $\mathcal{P}$ of dimension $k$ 
therefore as generalized $k$-dimensional points. Complete graphs $K_{k+1}=K_{(1,\dots,1)}$ 
are just a special case. The intersection of $\mathcal{P}$ and spheres $\mathcal{S}$ is
the set of cross polytopes, where every additive prime factor is a zero-dimensional sphere
$K_{(2)} = \overline{K_2}$. In short, elements in $\mathcal{P}$ can serve as $k$-dimensional
points. 

\section{Code}

\paragraph{}
Some Mathematica code appeared already in \cite{DiscreteAlgebraicSets}, where we considered
the case of maps $f: G \to \mathbb{R}^k$. The abstract version is much more elegant
because we do not have to program each dimension separately. It can serve also as 
pseudo code.  

\begin{tiny} \lstset{language=Mathematica} \lstset{frameround=fttt}
\begin{lstlisting}[frame=single]
Generate[A_]:=If[A=={},{},Sort[Delete[Union[Sort[Flatten[Map[Subsets,A],1]]],1]]];
Whitney[s_]:=Generate[FindClique[s,Infinity,All]];L=Length; Ver[X_]:=Union[Flatten[X]];
RFunction[G_,P_]:=Module[{},R[x_]:=x->RandomChoice[Range[L[Ver[P]]]];Map[R,Ver[G]]];
Facets[G_]:=Select[G,(L[#]==Max[Map[L,G]]) &];  Poly[X_]:=PolyhedronData[X,"Skeleton"];
AbstractSurface[G_,f_,A_]:=Select[G,(Sum[If[SubsetQ[#/.f,A[[l]]],1,0],{l,L[A]}]>0)&];
Z[n_]:=Partition[Range[n],1]; ZeroJoin[a_]:=If[L[a]==1,Z[a[[1]]],Whitney[CompleteGraph[a]]];
CJoin[G_,H_]:=Union[G,H+Max[G]+1,Map[Flatten,Map[Union,Flatten[Tuples[{G,H+Max[G]+1}],0]]]];
ToGraph[G_]:=UndirectedGraph[n=L[G];Graph[Range[n],
  Select[Flatten[Table[k->l,{k,n},{l,k+1,n}],1],(SubsetQ[G[[#[[2]]]],G[[#[[1]]]]])&]]];
J=Whitney[Poly["Icosahedron"]]; G=CJoin[J,Whitney[CycleGraph[13]]]; p={2,2,3};P=ZeroJoin[p];
V=Ver[P]; F=Facets[P]; H=AbstractSurface[G,RFunction[G,V],F];   GraphPlot3D[ToGraph[H]]
\end{lstlisting}
\end{tiny}

\begin{figure}[!htpb]
\scalebox{0.15}{\includegraphics{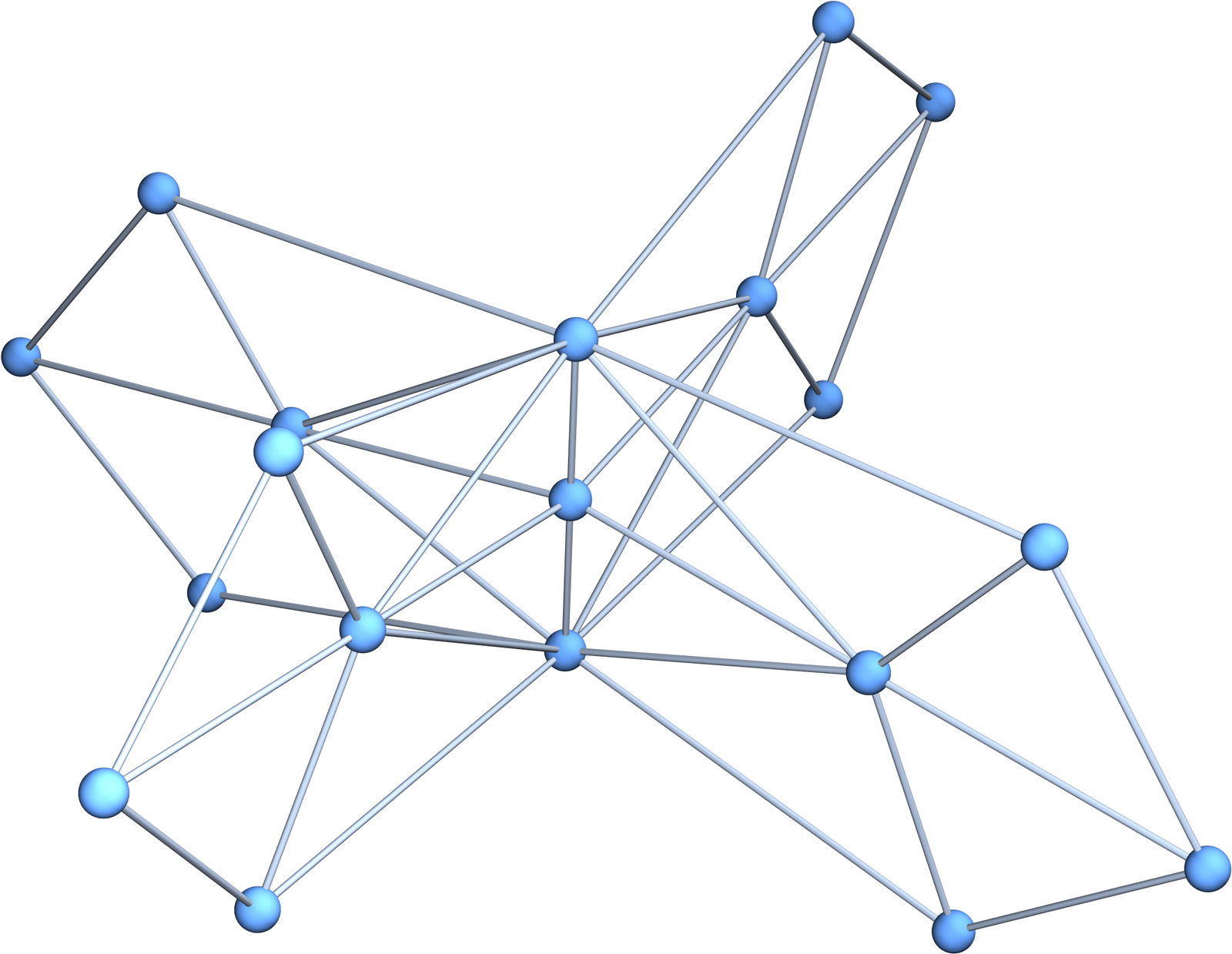}}
\scalebox{0.15}{\includegraphics{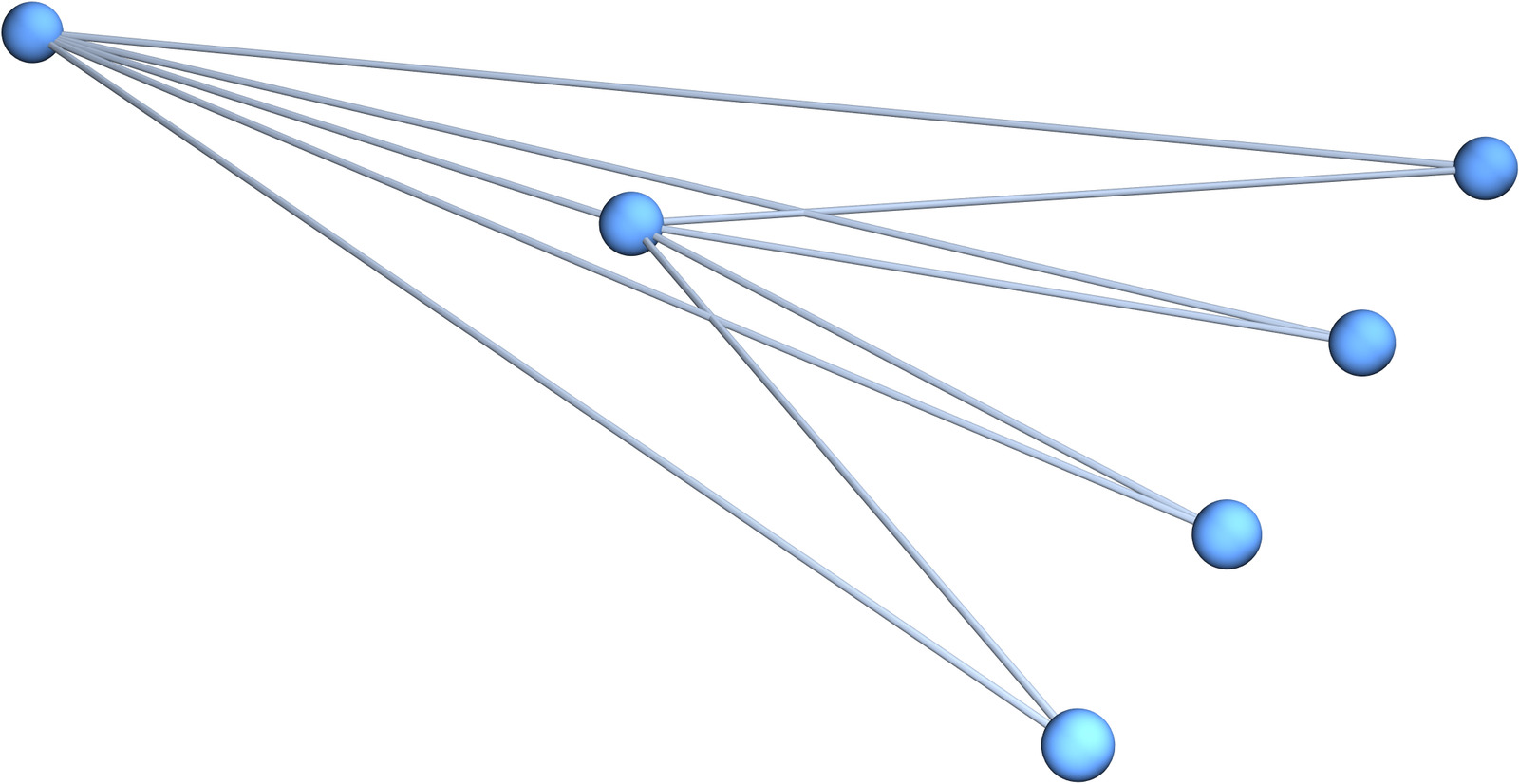}}
\caption{
The graph of the partition complex $K_{(1,1,4)}$ 
was generated with $ToGraph[ZeroJoin[\{1, 1, 4\}]]$
It is the Barycentric refinement of the 
$CompleteGraph[\{1,1,4\}]$. The procedure
$ZeroJoin[\{1,1,4\}]$ produces the simplicial complex
$\{\{1\},\{2\},\{3\},\{4\},\{5\},\{6\}$,
$\{1,2\},\{1,3\},\{1,4\},\{1,5\},\{1,6\},\{2,3\},\{2,4\},\{2,5\},\{2,6\}$,
$\{1,2,3\},\{1,2,4 \},\{1,2,5\},\{1,2,6\}\}$. 
}
\end{figure}

\paragraph{}
By the way, the addition in the monoid of partitions is already called {\it Join}
$Join[\{1,2,3\},\{1,6,6\}]$ gives $\{1,2,3,1,6,6\}$. This fits, as the
graph $K_{(1,2,3,1,6,6)}$ is the Zykov join of $K_{(1,2,3)}$ and $K_{(1,6,6)}$. 
The Zykov join is in the Wolfram language known as {\it GraphJoin}.

\begin{tiny} \lstset{language=Mathematica} \lstset{frameround=fttt}
\begin{lstlisting}[frame=single]
U=CompleteGraph[{1,1,2,3,6,6}];
V=GraphJoin[CompleteGraph[{1,2,3}],CompleteGraph[{1,6,6}]];
W=CompleteGraph[Join[{1,2,3},{1,6,6}]]; 
IsomorphicGraphQ[U,V,W] 
\end{lstlisting}
\end{tiny}

\section{Statistics}

\paragraph{}
The following code generates all surjective maps and analyzes in each 
case the topology.  This works only for very small host manifolds. To see the topological nature
of the manifolds, we look at the cohomology in the form of {\bf Betti vectors} 
$(b_0,b_1, \dots b_d)$ which gives the dimensions $b_k$ of the {\bf harmonic $k$-forms} of 
the complex. This in particular gives the number of connectivity components $b_0$ 
or the {\bf genus} $b_1/2$, which is for surfaces the ``number of holes".

\begin{tiny} \lstset{language=Mathematica} \lstset{frameround=fttt}
\begin{lstlisting}[frame=single]
Generate[A_]:=If[A=={},{},Sort[Delete[Union[Sort[Flatten[Map[Subsets,A],1]]],1]]];
Whitney[s_]:=Generate[FindClique[s,Infinity,All]];L=Length; Ver[X_]:=Union[Flatten[X]];
sig[x_]:=Signature[x]; nu[A_]:=If[A=={},0,L[NullSpace[A]]];
F[G_]:=Module[{l=Map[L,G]},If[G=={},{},Table[Sum[If[l[[j]]==k,1,0],{j,L[l]}],{k,Max[l]}]]];
sig[x_,y_]:=If[SubsetQ[x,y]&&(L[x]==L[y]+1),sig[Prepend[y,Complement[x,y][[1]]]]*sig[x],0];
Dirac[G_]:=Module[{f=F[G],b,d,n=L[G]},b=Prepend[Table[Sum[f[[l]],{l,k}],{k,L[f]}],0];
 d=Table[sig[G[[i]],G[[j]]],{i,n},{j,n}]; {d+Transpose[d],b}];
Beltrami[G_]:= Module[{B=Dirac[G][[1]]},B.B];
Hodge[G_]:=Module[{Q,b,H},{Q,b}=Dirac[G];H=Q.Q;Table[Table[H[[b[[k]]+i,b[[k]]+j]],
 {i,b[[k+1]]-b[[k]]},{j,b[[k+1]]-b[[k]]}],{k,L[b]-1}]];
Betti[s_]:=Module[{G},If[GraphQ[s],G=Whitney[s],G=s];Map[nu,Hodge[G]]];
Facets[G_]:=Select[G,(L[#]==Max[Map[L,G]]) &];  Poly[X_]:=PolyhedronData[X,"Skeleton"];
AbstractSurface[G_,f_,A_]:=Select[G,(Sum[If[SubsetQ[#/.f,A[[l]]],1,0],{l,L[A]}]>0)&];
Z[n_]:=Partition[Range[n],1]; ZeroJoin[a_]:=If[L[a]==1,Z[a[[1]]],Whitney[CompleteGraph[a]]];
CJoin[G_,H_]:=Union[G,H+Max[G]+1,Map[Flatten,Map[Union,Flatten[Tuples[{G,H+Max[G]+1}],0]]]];
ComputeAllBetti[G_,P_]:=Module[{V=Ver[P],Fa=Facets[P],n=Length[Ver[G]]}, 
  A=Tuples[V,n]; Surjectiv=Select[A,(L[Union[#]]==L[V]) &];  Print[Length[Surjectiv]]; 
  Phi[g_]:=Module[{f=Table[k->g[[k]],{k,n}]},H=AbstractSurface[G,f,Fa];Length[H]]; 
  AllSurfaces=Map[Phi,Surjectiv]; B=Flatten[Position[AllSurfaces,0]];
  EmptyCases=Table[Surjectiv[[B[[k]]]],{k,Length[B]}];
  NonEmptyCases=Complement[Surjectiv,EmptyCases]; 
  Psi[g_]:=Module[{f=Table[k->g[[k]],{k,n}]},H=AbstractSurface[G,f,Fa];Betti[H]];
  AllBetti=Map[Psi,NonEmptyCases]];  

S3 = ZeroJoin[{2,2,2,2}]; S4 = ZeroJoin[{2,2,2,2,2}]; S5 = ZeroJoin[{2,2,2,2,2,2}]; 
RP3=Generate[{{1,2,3,4},{1,2,3,5},{1,2,4,6},{1,2,5,6},{1,3,4,7},{1,3,5,7},{1,4,6,7},
{1,5,6,7},{2,3,4,8},{2,3,5,9},{2,3,8,9},{2,4,6,10},{2,4,8,10},{2,5,6,11},{2,5,9,11},
{2,6,10,11},{2,7,8,9},{2,7,8,10},{2,7,9,11},{2,7,10,11},{3,4,7,11},{3,4,8,11},
{3,5,7,10},{3,5,9,10},{3,6,8,9}, {3,6,8,11},{3,6,9,10},{3,6,10,11},{3,7,10,11},
{4,5,8,10},{4,5,8,11},{4,5,9,10},{4,5,9,11},{4,6,7,9},{4,6,9,10},{4,7,9,11},{5,6,7,8},
{5,6,8,11},{5,7,8,10},{6,7,8,9}}];
CP2=Generate[{{1,2,3,4,5},{1,2,3,4,7},{1,2,3,5,8},{1,2,3,7,8},{1,2,4,5,6},
  {1,2,4,6,7},{1,2,5,6,8},{1,2,6,7,9},{1,2,6,8,9},{1,2,7,8,9},{1,3,4,5,9},{1,3,4,7,8},
  {1,3,4,8,9},{1,3,5,6,8},{1,3,5,6,9},{1,3,6,8,9},{1,4,5,6,7},{1,4,5,7,9},{1,4,7,8,9},
  {1,5,6,7,9},{2,3,4,5,9},{2,3,4,6,7},{2,3,4,6,9},{2,3,5,7,8},{2,3,5,7,9},{2,3,6,7,9},
  {2,4,5,6,8},{2,4,5,8,9},{2,4,6,8,9},{2,5,7,8,9},{3,4,6,7,8},{3,4,6,8,9},{3,5,6,7,8},
  {3,5,6,7,9},{4,5,6,7,8},{4,5,7,8,9}}]; 
S2xS2=Generate[{{1,2,3,4,6},{1,2,3,4,7},{1,2,3,6,9},{1,2,3,7,9},{1,2,4,5,8},{1,2,4,5,9},
  {1,2,4,6,8},{1,2,4,7,9},{1,2,5,6,8},{1,2,5,6,9},{1,3,4,6,7},{1,3,5,6,7},{1,3,5,6,9},
  {1,3,5,7,10},{1,3,5,9,11},{1,3,5,10,11},{1,3,7,9,10},{1,3,9,10,11},{1,4,5,8,10},
  {1,4,5,9,11},{1,4,5,10,11},{1,4,6,7,11},{1,4,6,8,10},{1,4,6,10,11},{1,4,7,9,11},
  {1,5,6,7,8},{1,5,7,8,10},{1,6,7,8,11},{1,6,8,10,11},{1,7,8,10,11},{1,7,9,10,11},
  {2,3,4,6,8},{2,3,4,7,8},{2,3,5,7,10},{2,3,5,7,11},{2,3,5,10,11},{2,3,6,8,10},{2,3,6,9,10},
  {2,3,7,8,11},{2,3,7,9,10},{2,3,8,10,11},{2,4,5,8,9},{2,4,7,8,9},{2,5,6,8,11},{2,5,6,9,10},
  {2,5,6,10,11},{2,5,7,8,9},{2,5,7,8,11},{2,5,7,9,10},{2,6,8,10,11},{3,4,6,7,11},{3,4,6,8,10},
  {3,4,6,9,10},{3,4,6,9,11},{3,4,7,8,11},{3,4,8,9,10},{3,4,8,9,11},{3,5,6,7,11},{3,5,6,9,11},
  {3,8,9,10,11},{4,5,6,9,10},{4,5,6,9,11},{4,5,6,10,11},{4,5,8,9,10},{4,7,8,9,11},{5,6,7,8,11},
  {5,7,8,9,10},{7,8,9,10,11}}]; 
P=ZeroJoin[{1,1,1}];A1=ComputeAllBetti[S3,P];   Save["allbetti_2222_111.txt", A1]; 
P=ZeroJoin[{1,1,1}];A2=ComputeAllBetti[S4,P];   Save["allbetti_22222_111.txt",A2]; 
P=ZeroJoin[{1,1}];  A3=ComputeAllBetti[S4,P];   Save["allbetti_22222_11.txt", A3]; 
P=ZeroJoin[{1,1}];  A4=ComputeAllBetti[RP3,P];  Save["allbetti_rp3_11.txt",   A4];
P=ZeroJoin[{1,1,1}];A5=ComputeAllBetti[CP2,P];  Save["allbetti_cp2_111.txt",  A5];
P=ZeroJoin[{1,1}];  A6=ComputeAllBetti[CP2,P];  Save["allbetti_cp2_11.txt",   A6];
P=ZeroJoin[{1,1}];  A7=ComputeAllBetti[S2xS2,P];Save["allbetti_s2xs2_11.txt", A7];
P=ZeroJoin[{1,1,1}];A8=ComputeAllBetti[S2xS2,P];Save["allbetti_s2xs2_111.txt",A8];
P=ZeroJoin[{1,2}];  A9=ComputeAllBetti[S3,P];   Save["allbetti_2222_12.txt",  A9]; 
P=ZeroJoin[{1,3}];  A10=ComputeAllBetti[S3,P];  Save["allbetti_2222_13.txt", A10]; 
\end{lstlisting}
\end{tiny}

\paragraph{}
{\bf Example 1}: In the list of all co-dimension 2 curves in the smallest 3 sphere $K_{(2,2,2,2)}$. 
There were 4848 single circles,888 cases with 2 circles and 36 cases with 4 circles.
There were 5772 surjective maps from $G$ to $K_3$ which produced a non-empty surfaces.
For only 24 cases, the resulting manifold is empty.

\paragraph{}
{\bf Example 2}: Take the smallest 4 sphere $K_{(2,2,2,2,2)}$. 
There were 55980 surjective maps from $G$ to $K_3$. This produces co-dimension 2 surfaces. 
Only 30 surjective function sproduced an empty 
surface. The other 55950 produced either 2-spheres (50160), 2 disjoint 2-spheres (3960) or 4 disjoint 
2-spheres (60 cases) or single $2$-tori (1680 cases) or 2 disjoint $2$-tori (90 cases). 

\paragraph{}
{\bf Example 3}: For all co-dimension $1$ surfaces in the 4-sphere $G=K_{(2,2,2,2,2)}$, we see
$992$ 3-spheres, $20$ cylinders $S^2 \times S^1$ and $10$ pairs of disjoint sphere. 
There were $1022$ surjective maps from $G$ to $K_2$, so that 30 produced empty sets.

\paragraph{}
{\bf Example 4}: All co-dimension 1 surfaces in a small projective 3-manifold $\mathbb{RP}^3$
produce $2046=2^{12}-2$ possible surjective maps to $K_2$. All of them give surfaces. 
There are $548$ spheres, $1394$ tori, $68$ genus 2 surfaces, $8$ genus 3 surfaces, $20$ pairs of spheres,
$8$ pairs with one sphere and one torus and 2 pairs, where one is a $2$-torus and one a genus 2 surface.

\paragraph{}
{\bf Example 5}: When looking at co-dimension 2 surfaces in a small implementation of 
the 4-manifold $\mathbb{CP}^2$, we count $8748$ 2-spheres and $9402$ tori. 

\paragraph{}
{\bf Example 6}: Among all co-dimension 1 surfaces in a small implementation of
$\mathbb{CP}^2$ we count $511$ 3-spheres. 

\paragraph{}
{\bf Example 7} For all hyper-surfaces defined by $K_2$ in the small 4-manifold $\mathbb{S}^2 \times \mathbb{S}^2$
defined by surjective to $K_2$ the nuumber of $3$-manifolds is 2046.
There were 1446 3-spheres with Betti vector $(1,0,0,1)$ and 600 generalized tori $S^2 \times S^1$
with Betti vector $(1,1,1,1)$.

\paragraph{}
{\bf Example 8} looks for co-dimension 2 surfaces in a small $G=S^2 \times S^2$ using the triangle 
$P=K_3$. There are $42288$ two-spheres, $71022$  genus 1 surfaces ($2$-tori),
$48210$ genus 2 surfaces, $7788$ genus 3 surfaces,
$660$ genus 4 surfaces, $78$ genus 5 surfaces and 960 disconnected pairs of 2-spheres. 
In total, the cohomology of $171006$ surfaces was computed. 

\paragraph{}
{\bf Example 9}  For all hyper-surfaces in the small 3 sphere $K_{(2,2,2,2)}$ with $P=K_{(1,2)}$ 
we have $5496$ surfaces, where $5456$ were 2-spheres, $84$ were $2$-tori and $256$ were double spheres. 

\paragraph{}
{\bf Example 10}  For all hyper-surfaces in the small 3 sphere $G=K_{(2,2,2,2)}$ with $P=K_{(1,3)}$ we have $40824$ 
cases in total, no empty surfaces, $38448$ spheres, $216$ tori and $2160$ double spheres. 
In this case we look at all maps from $G$ to $\{1,2,3,4\}$ because $P$ has $4$ vertices. This requires to look at 
much more maps. An interesting question is how large the total number of hyper-surfaces $\bigcup_P \bigcup_f H(G,f,P)$ 
can become. As they are all sub-graphs of the Barycentric refinement $G_1$ of $G$, there are only finitely many. 
Many $H(G,f,P)$ are are actually the same graph. 

\paragraph{}
As the gallery in the next section illustrates, much more interesting sub-manifolds
are possible for larger host manifolds $G$.
But the number of possible maps is too large for an exhaustive list. 
We would have to use Monte-Carlo simulations to investigate what manifolds typically occur. 

\pagebreak 

\section{Gallery}

\begin{figure}[!htpb]
\scalebox{0.075}{\includegraphics{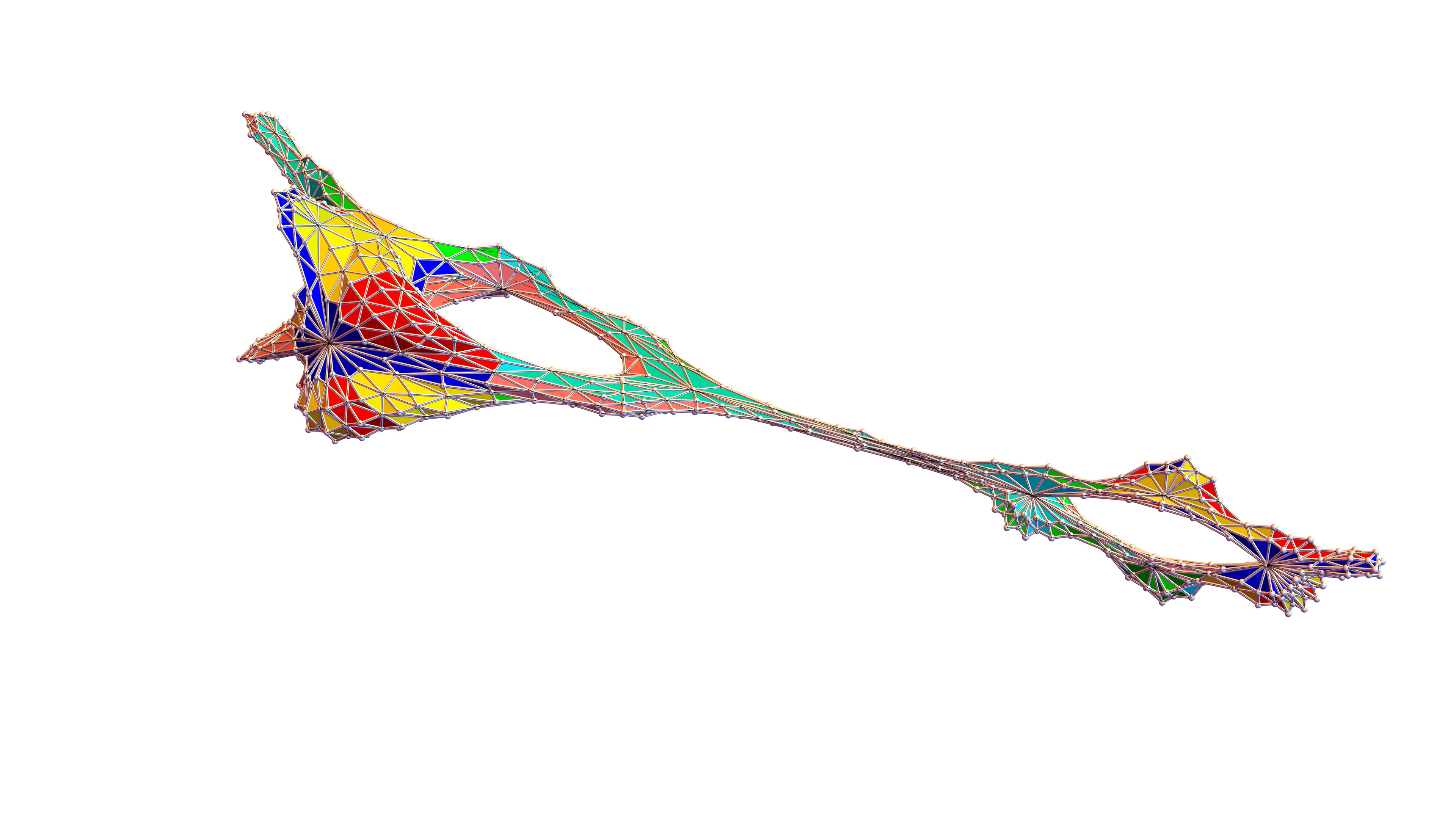}}
\scalebox{0.075}{\includegraphics{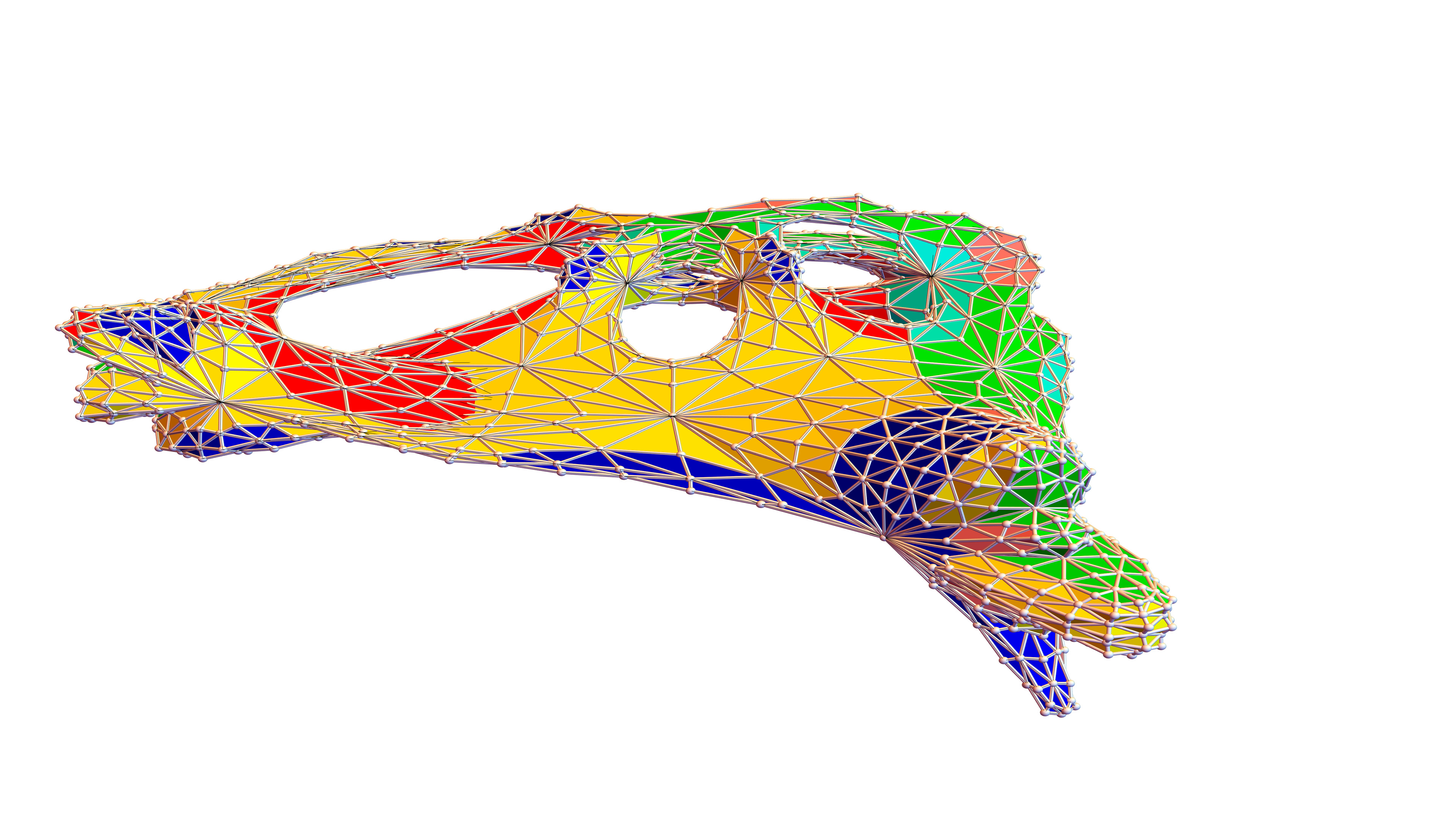}}
\scalebox{0.075}{\includegraphics{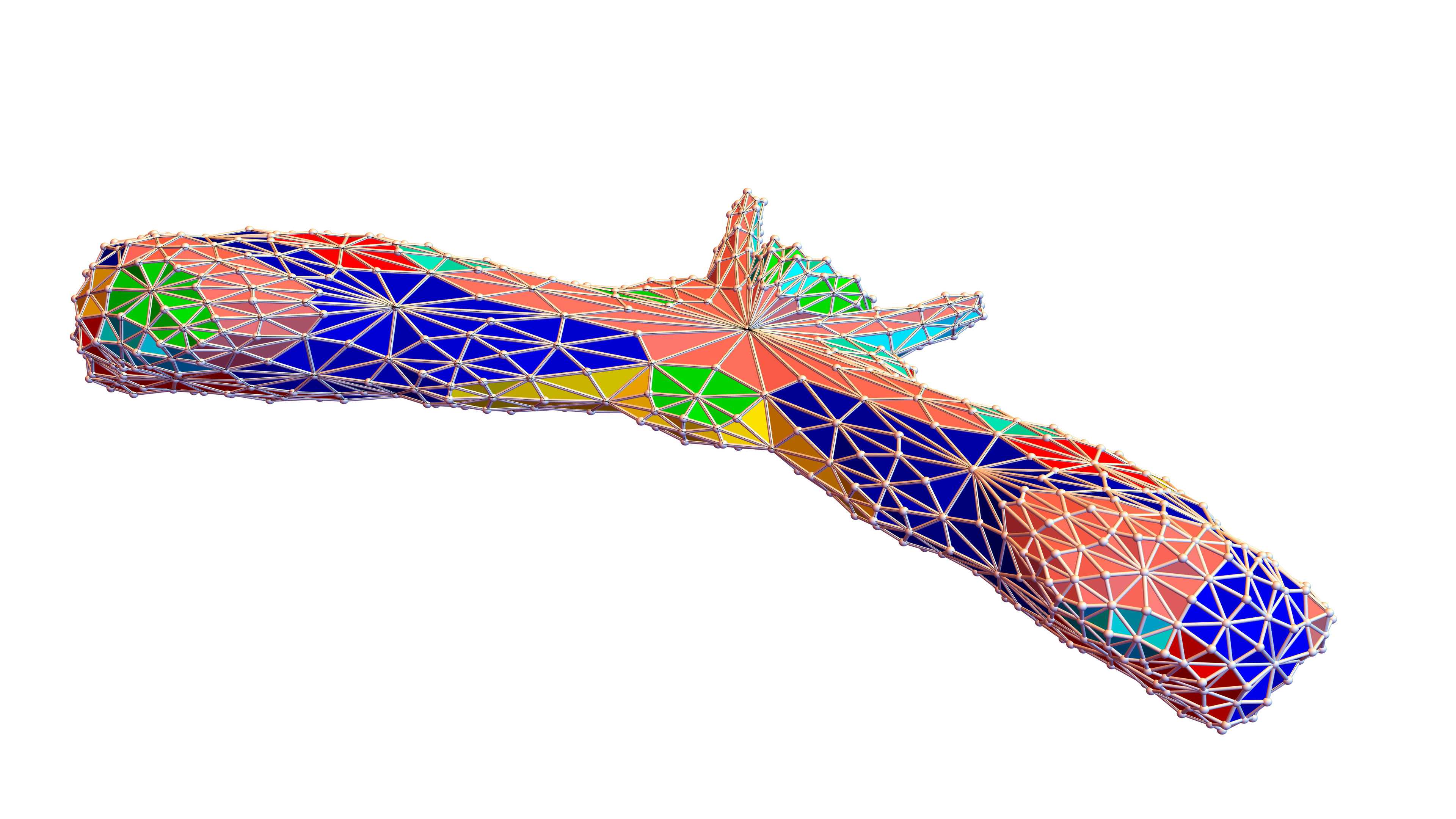}}
\scalebox{0.075}{\includegraphics{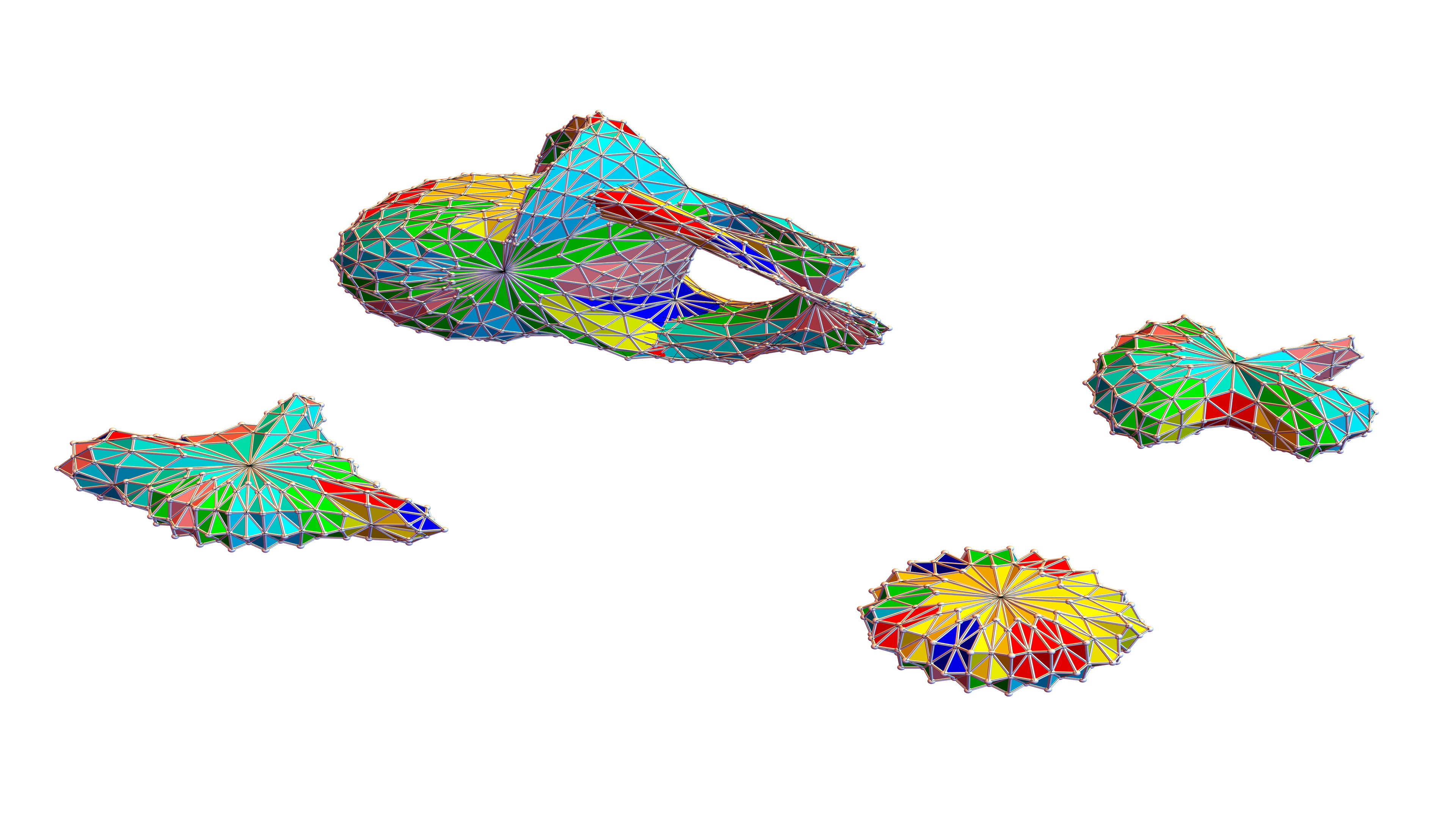}}
\scalebox{0.075}{\includegraphics{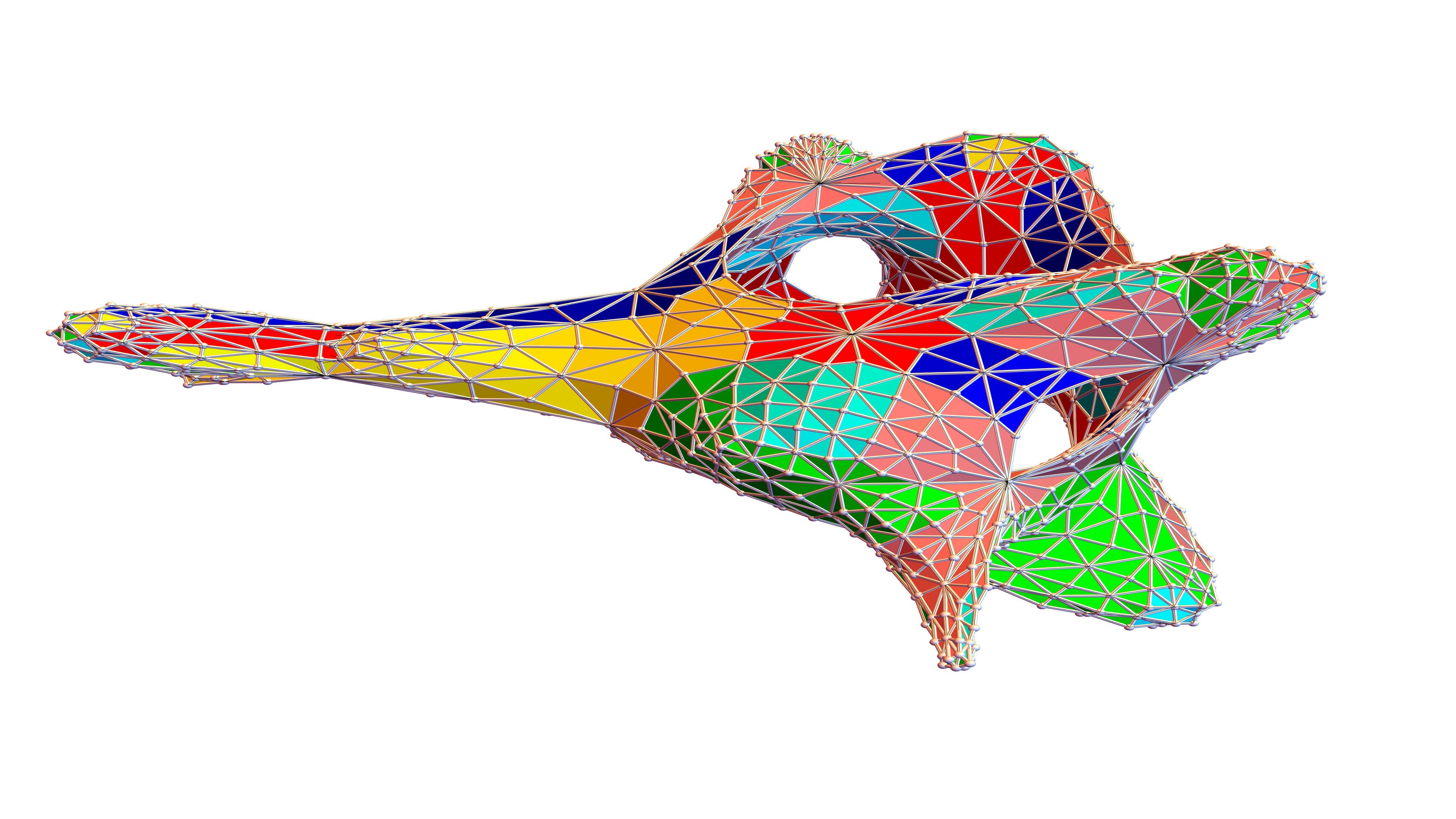}}
\scalebox{0.075}{\includegraphics{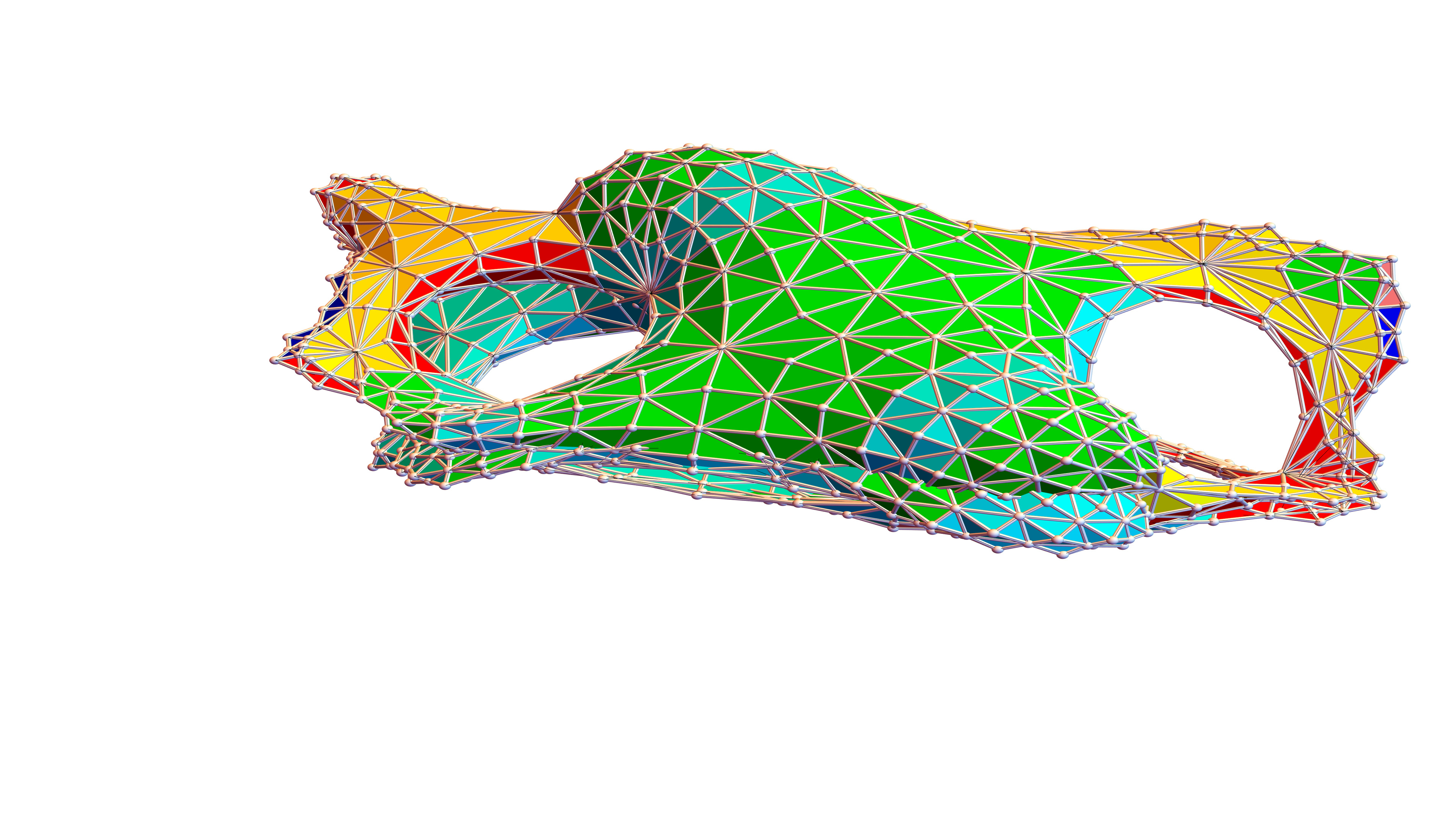}}
\scalebox{0.075}{\includegraphics{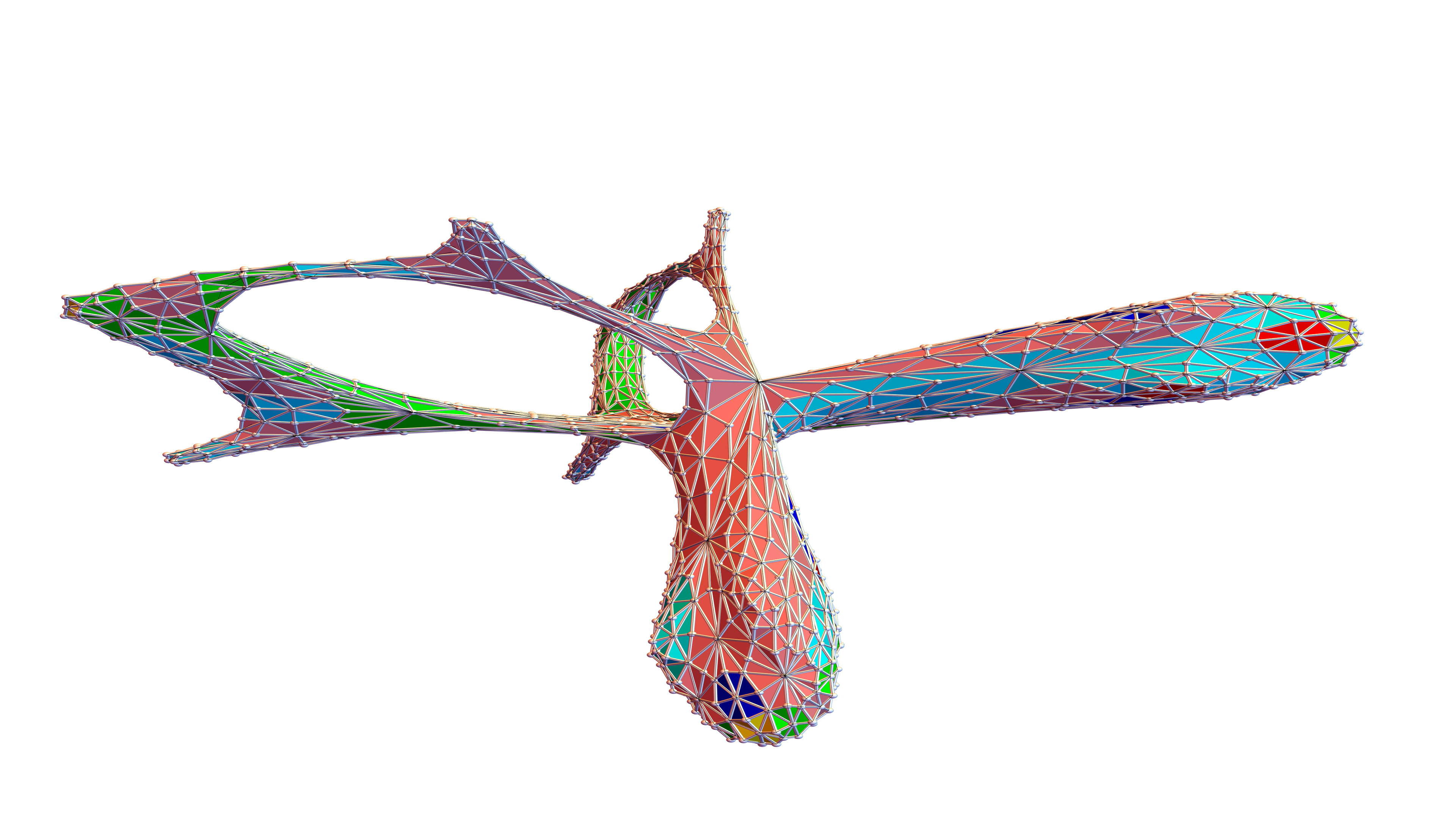}}
\scalebox{0.075}{\includegraphics{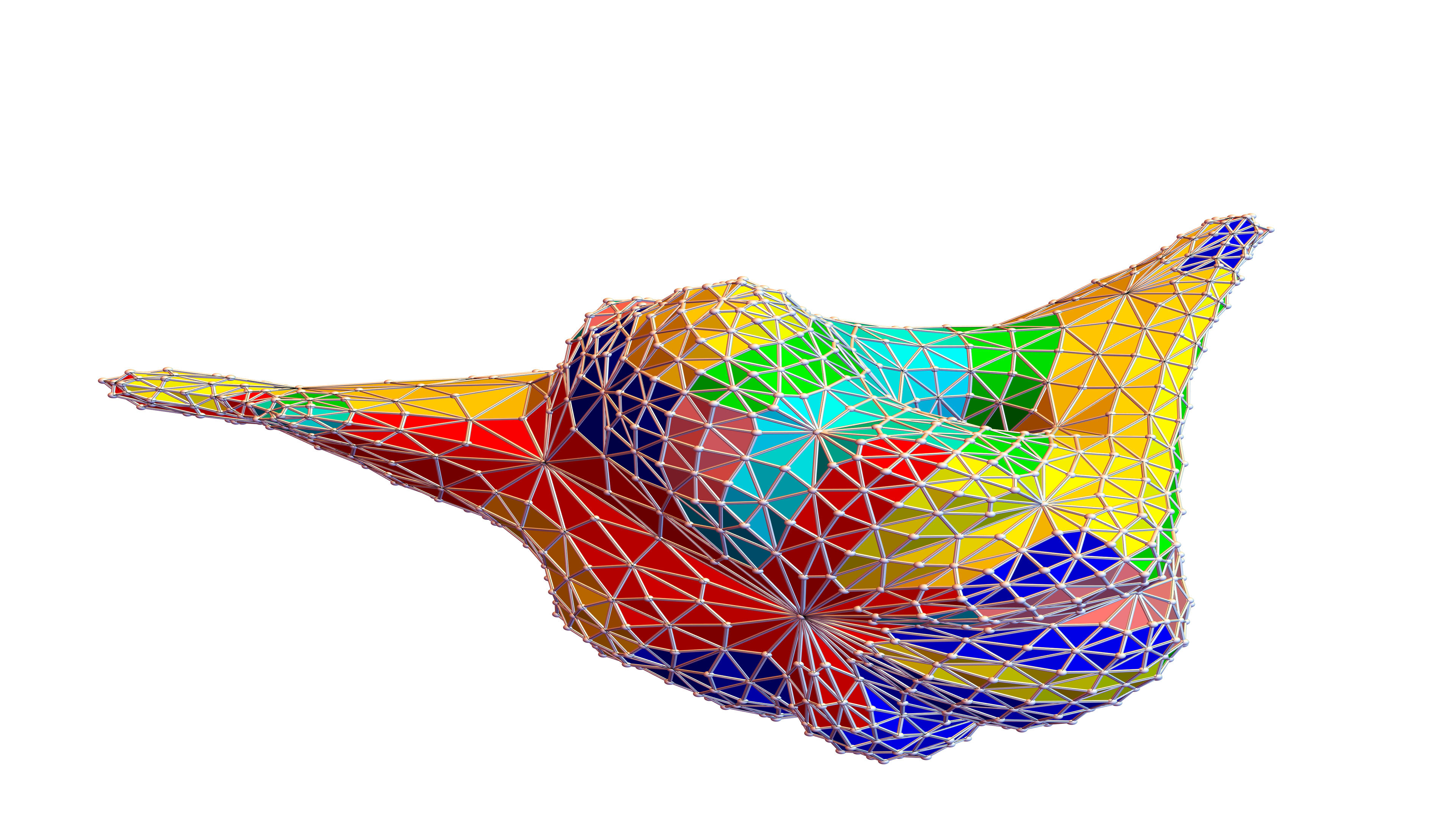}}
\caption{
The host $G$ is a 4-sphere defined as the join of $C_{15}$ and the Barycentric 
refinement ${\rm Oct}_1$ of the Octahedron $K_{\{2,2,2\}}={\rm Oct}$. The 4-manifold 
$G$ has $26+15 = 41$ vertices because ${\rm Oct}_1$ has $6+8+12=26$ vertices.
The partition complex $P$ is defined by $p=(1,2,3)$. There are already
$\sum_{j=0}^6 \Binomial{41}{6-j} j^{41}$ 
$=7834446798231431881744478798427378344467982314318817444787984273$ 
surjective functions from $\{1,\dots,24\}$ onto $\{1,\dots,6\}$.
The number of functions from $G$ onto $P$ (as defined in the first page) is
even larger, as a function must only cover only at least one of the 6 triangles
in $P$. Each of these $f$'s produces a 2-manifolds $H(G,f,P)$,
colored depending which of the triangles in $P$ is reached. $H(G,f,P)$ is
so partitioned into sub-manifolds with boundary. 
}
\end{figure}

\pagebreak 

\bibliographystyle{plain}

\begin{thebibliography}{10}

\bibitem{Alexandroff1937}
P.~Alexandroff.
\newblock Diskrete {R\"aume}.
\newblock {\em Mat. Sb. 2}, 2, 1937.

\bibitem{CuisenaireGattegno}
G.~Cuisenaire and C.~Gattegno.
\newblock {\em Numbers in Colour: A New Method of Teaching the Processes of
  Arithmetic to All Levels of the Primary School}.
\newblock Heinemann, 1957.

\bibitem{HammackImrichKlavzar}
R.~Hammack, W.~Imrich, and S.~Klav\v zar.
\newblock {\em Handbook of product graphs}.
\newblock Discrete Mathematics and its Applications (Boca Raton). CRC Press,
  Boca Raton, FL, second edition, 2011.
\newblock With a foreword by Peter Winkler.

\bibitem{KnillSchweizerJugendForscht}
O.~Knill.
\newblock {A}nschauliche additive {Z}ahlentheorie.
\newblock In B.~Roethlin, editor, {\em Schweizer Jugend forscht year book 83.},
  pages 200--210. {Winterthur: Verlag Schweizer Jugend forscht}, 1983.

\bibitem{knillcalculus}
O.~Knill.
\newblock {The theorems of Green-Stokes,Gauss-Bonnet and Poincare-Hopf in Graph
  Theory}.
\newblock http://arxiv.org/abs/1201.6049, 2012.

\bibitem{KnillSard}
O.~Knill.
\newblock A {S}ard theorem for graph theory.
\newblock {{\\}http://arxiv.org/abs/1508.05657}, 2015.

\bibitem{dehnsommervillegaussbonnet}
O.~Knill.
\newblock Dehn-{S}ommerville from {G}auss-{B}onnet.
\newblock {\\}https://arxiv.org/abs/1905.04831, 2019.

\bibitem{DiscreteAlgebraicSets}
O.~Knill.
\newblock {Cohomology of Open sets}.
\newblock https://arxiv.org/abs/2305.12613, 2023.

\bibitem{FiniteTopology}
O.~Knill.
\newblock Finite topologies for finite geometries.
\newblock https://arxiv.org/abs/2301.03156, 2023.

\bibitem{Morse1939}
A.P. Morse.
\newblock The behavior of a function on its critical set.
\newblock {\em Ann. of Math. (2)}, 40(1):62--70, 1939.

\bibitem{Sabidussi}
G.~Sabidussi.
\newblock Graph multiplication.
\newblock {\em Math. Z.}, 72:446--457, 1959/1960.

\bibitem{Sard1942}
A.~Sard.
\newblock The measure of the critical values of differentiable maps.
\newblock {\em Bull. Amer. Math. Soc.}, 48:883--890, 1942.

\bibitem{Shannon1956}
C.~Shannon.
\newblock The zero error capacity of a noisy channel.
\newblock {\em IRE Transactions on Information Theory}, 2:8--19, 1956.

\bibitem{Zykov}
A.A. Zykov.
\newblock On some properties of linear complexes. (russian).
\newblock {\em Mat. Sbornik N.S.}, 24(66):163--188, 1949.

\end{thebibliography}

\end{document}